\documentclass[11pt, reqno]{amsart}
\usepackage{fullpage}

\newif\ifArxiv
\Arxivtrue
 \newif\ifHideFoot
\HideFoottrue
\ifHideFoot
\else
\usepackage[notref,notcite]{showkeys}
\fi 

\usepackage{amssymb,amsmath,amsthm,amscd,mathrsfs,graphicx, color}
\usepackage[cmtip,all,matrix,arrow,tips,curve]{xy}
 \usepackage{hyperref}
\usepackage[usenames,dvipsnames]{xcolor}
\usepackage{xypic}
\hypersetup{colorlinks=true,citecolor=ForestGreen,linkcolor=Maroon,urlcolor=NavyBlue}
 \usepackage{mathtools}
\usepackage{mathpazo}
\usepackage[normalem]{ulem}
\usepackage{thmtools}
\usepackage{cleveref}
\usepackage{enumitem}
\usepackage{tikz-cd}
  
\numberwithin{equation}{section}
\newtheorem{teo}{Theorem}[section]
\newtheorem{pro}[teo]{Proposition}
\newtheorem{lem}[teo]{Lemma}
\newtheorem{cor}[teo]{Corollary}

\newtheorem{teoalpha}{Theorem}

\theoremstyle{definition}

\newtheorem{dfn}[teo]{Definition}
\theoremstyle{remark}
\newtheorem{rem}[teo]{Remark}

\ifHideFoot

\newcommand{\Yano}[1]{}
\newcommand{\Shend}[1]{}

\else 

\newcommand{\marg}[1]{\normalsize{{
\color{red}\footnote{{\color{blue}#1}}}{\marginpar[\vskip
-.25cm{\color{red}\hfill\tiny\thefootnote$\implies$}]{\vskip
-.2cm{\color{red}$\impliedby$\tiny\thefootnote}}}}}
\newcommand{\Yano}[1]{\marg{(Yano) #1}}
\newcommand{\Shend}[1]{\marg{(Shend) #1}}

\fi

\newcommand{\m}[1]{\mathcal{#1}}

\newcommand{\X}{\mathcal{X}}

\newcommand{\ra}{\rightarrow}

\title[Positivity for Hodge modules on stacks]{Movable curve classes and slope stability on Deligne--Mumford stacks}

\author{Sebastian Casalaina-Martin}
\address{University of Colorado, Department of Mathematics, 
Boulder, CO 80309, USA }
\email{casa@math.colorado.edu}

\author{Shend Zhjeqi}
\address{University of Michigan, Department of Mathematics, 
Ann Arbor, MI 48109, USA }
\email{shendzh@umich.edu}

\thanks{Research of the first named author is supported in part by a grant from the Simons Foundation (SFI-MPS-TSM-00013682). The second named author was partially supported by the Simons Collaboration grant Moduli of Varieties.}

\date{\today}

\begin{document}

\begin{abstract}
We generalize some results in the literature on movable curve classes and slope stability of coherent sheaves on smooth projective varieties to the case of smooth proper DM stacks admitting projective coarse moduli spaces.  As an application, we establish a Bogomolov--Gieseker inequality on such stacks.  This paper is the second in a series aiming to generalize results of Popa--Schnell and Wei--Wu on Viehweg hyperbolicity to the setting of DM stacks, and in particular, to certain KSBA moduli spaces.
\end{abstract}

\maketitle

\section*{Introduction}

Slope stability of coherent sheaves with respect to an ample line bundle on a smooth complex projective variety  plays a central role in the study of coherent sheaves. Among many applications are constructions of moduli spaces of sheaves, as well as various constraints on invariants of stable sheaves, such as the Bogomolov--Gieseker inequality \cite{G79bogomolov}.  In order to apply these techniques in the setting of birational geometry, where the pull-back of an ample line bundle need not be ample, it has become clear recently that one should consider, instead, slope stability with respect to movable classes of curves, i.e., classes that pair non-negatively with all pseudo-effective divisors \cite{CPT11, GKP14, GKP16, GKPT19,CPFol19}. 
Recall that for a nonzero torsion-free coherent sheaf $\mathcal E$ on a smooth complex projective variety $X$ and a movable curve class $\alpha\in \operatorname{N}_1(X)_{\mathbb R}$, one defines the $\alpha$-slope of $\mathcal E$ as
\begin{equation}\label{E:introSlp}
\mu_\alpha(\mathcal E):=\frac{c_1(\mathcal E)\cdot \alpha}{\operatorname{rk}(\mathcal E)}.
\end{equation}
The sheaf $\mathcal E$ is said to be $\alpha$-semi-stable (resp.~$\alpha$-stable) if for all non-zero subsheaves (resp.~non-zero subsheaves of smaller rank than $\mathcal E$) $\mathcal F\subseteq \mathcal E$, one has $\mu_\alpha(\mathcal F)\le \mu_\alpha(\mathcal E)$ (resp.~$\mu_\alpha(\mathcal F)< \mu_\alpha(\mathcal E)$).  

In order to  apply these techniques elsewhere, to study the birational geometry of moduli spaces, in this paper we extend the notion of movable classes of  curves, and slope stability with respect to such classes, to the case of smooth proper Deligne--Mumford stacks with projective coarse moduli spaces.  We note that the case where $\mathcal X=[V/G]$ is a finite quotient stack, i.e., $\mathcal X$ is the quotient of a smooth projective variety $V$ by a finite group $G$ acting on $V$, has been considered in \cite{CPFol19}, and the results in this paper should be considered a generalization of those results to smooth proper DM stacks over $\mathbb C$ with projective coarse moduli space.  

We define movable curve classes on such DM stacks in  \Cref{S:mov-class-XX} (\Cref{D:MoveClass}), and define the slope of a coherent sheaf by \eqref{E:introSlp},  using the intersection theory from Vistoli \cite{vistoli89}.
We then define $\alpha$-(semi-)stability exactly as before (\Cref{S:SlopeStab}).  
Our main result is the following, generalizing the results in the literature for smooth projective varieties, and finite quotient stacks:

\begin{teoalpha}\label{T:main}
Let $\mathcal X$ be a smooth proper integral Deligne--Mumford stack over $\mathbb C$ with projective coarse moduli space.  For a movable class of curves $\alpha\in \operatorname{N}_1(\mathcal X)_{\mathbb R}$,  and non-zero torsion-free coherent sheaves on $\mathcal X$, then with respect to the $\alpha$-slope $\mu_\alpha(-)$ and $\alpha$-(semi-)stability we have the following:
\begin{enumerate} [label=(\alph*)]
\item \label{T:main1} Maximally destabilizing sub-sheaves exist (\Cref{P:GKP16-2.24});
\item \label{T:main2} Harder--Narisimhan filtrations exist and are unique 
(\Cref{L:HNFil});
\item \label{T:main3} Jordan--Holder filtrations of $\alpha$-semistable sheaves exist (\Cref{cor:JH});
\item \label{T:main4} Tensor products of $\alpha$-semi-stable sheaves are $\alpha$-semi-stable (\Cref{T:GKP16-4.2}).
\end{enumerate}
\end{teoalpha}

For the most part, \Cref{T:main} is a fairly straightforward extension of the standard treatment of slope stability in say \cite{huybrechts_lehn_2010}.  However, as with the case of smooth projective varieties, the most delicate part of the presentation is showing \ref{T:main4},  and in particular, showing that the tensor product of two semi-stable vector bundles is semi-stable.  
In the proof of the theorem on smooth projective varieties in \cite[Thm.~5.1, Prop.~5.2]{CPT11}, which is closely modeled on the argument in \cite[Thm.~3.1.4]{huybrechts_lehn_2010} for slope stability with respect to an ample line bundle, a central step is to reduce to the existence of a certain finite cover, and thereby reduce to showing  that semi-stability is equivalent to semi-stability after pull-back by a finite cover (e.g., \cite[Lem.~3.2.2]{huybrechts_lehn_2010}).  For projective varieties, this assertion about stability and finite covers is relatively easy to prove, using the elementary fact that for normal projective varieties, any finite cover is dominated by a finite \emph{Galois} cover.    The stability argument for Galois covers follows quickly, as one can easily descend destabilizing sub-sheaves from the Galois cover.  

This naturally leads to an approach for the DM stacks we consider here, since they all admit a finite flat cover $q:V\to \mathcal X$ from a smooth projective variety $V$.  In other words, if one knew that slope semi-stability for DM stacks was equivalent to slope semi-stability after pull back by a finite cover, one could easily reduce the proof of \Cref{T:main}\ref{T:main4}  to the case of varieties.  Unfortunately, for DM stacks, it is not the case that every finite cover $q:V\to \mathcal X$ from a normal projective variety $V$ is dominated by a Galois cover of $\mathcal X$ from a normal projective variety, and so one must take a different approach than in the case of varieties.  

This is the main complication arising for the  stacks appearing here that does not arise in the case $\mathcal X=[V/G]$ considered in \cite{CPFol19}.  Consequently, the main technical point in proving \Cref{T:main}\ref{T:main4} is to establish that slope semi-stability for DM stacks is equivalent to slope semi-stability after pull back by a finite cover to a smooth projective variety; we prove this in \Cref{P:StabOnV}.  While one could approach this via faithfully flat descent, for our proof of  \Cref{P:StabOnV} we instead use a result of Rydh \cite[Thm.~C]{rydhComact},  which implies that after a stacky blow-up $\mathcal X'\to \mathcal X$, the finite cover $q:V\to \mathcal X$ is dominated by a Galois cover in the sense that there is a smooth DM stack $\mathcal V'$ with projective coarse moduli space, with $\mathcal V'$ dominating $V$, and a finite group $G$ acting on $\mathcal V'$, such that $\mathcal X'=[\mathcal V'/G]$; see \Cref{T:Rydh}.  This turns out to be enough to prove \Cref{P:StabOnV}.

\medskip 

For applications of $\alpha$-stability,  a useful tool is the so-called maximal $\alpha$-slope,  $\mu_\alpha^{\max}(\mathcal E)$, of a nonzero torsion-free coherent  sheaf $\mathcal E$, which is defined as the supremum of the $\alpha$-slopes of all nonzero subsheaves $0\subsetneq \mathcal F\subseteq \mathcal E$. We establish some of the basic properties of $\mu_\alpha^{\max}(-)$ on DM stacks, as well  (generalizing, e.g., \cite[Cor.~2.24, Thm.~4.2]{GKP16}, \cite[Lem.~2.5]{CPFol19}): 

\begin{teoalpha}\label{T:max}
Let $\mathcal X$ be a smooth proper integral Deligne--Mumford stack over $\mathbb C$ with projective coarse moduli space.  For a movable class of curves $\alpha\in \operatorname{N}_1(\mathcal X)_{\mathbb R}$,  and non-zero torsion-free coherent sheaves $\mathcal E$ and $\mathcal E'$ on $\mathcal X$, we have:
\begin{enumerate} [label=(\alph*)]
\item \label{T:max1} $\mu_\alpha^{\max}(\mathcal E)$ is a maximum (\Cref{P:GKP16-2.24});
\item  \label{T:max2} If $\mu_\alpha^{\max}(\mathcal E)<0$, then $H^0(\mathcal X,\mathcal E)=0$ (\Cref{L:CP2.5});
\item \label{T:max3} $\mu_\alpha^{\max}(\mathcal E\otimes \mathcal E')=\mu_\alpha^{\\max}(\mathcal E)+\mu_\alpha^{\max}(\mathcal E')$ and $\mu_\alpha^{\max}(\operatorname{Sym}^m(\mathcal E))=m\cdot \mu_\alpha^{\max}(\mathcal E)$ (\Cref{T:GKP16-4.2} and \Cref{T:GKP16-4.2Sym}).
\end{enumerate}

\end{teoalpha}

The  proof of \Cref{T:max}\ref{T:max1} and \ref{T:max2} are fairly straightforward generalizations of the proofs in  \cite{GKP16,CPFol19} for smooth projective varieties.  As in the case of projective varieties, \Cref{T:max}\ref{T:max3}  is the most involved to prove; however, the central piece of the argument is \Cref{T:main}\ref{T:main4}, showing that the tensor product of semi-stable vector bundles is semi-stable; with that result, the proof of \Cref{T:main}\ref{T:main4} follows that of \cite{GKP16}.

\medskip 

Finally, as an  application of the results in this paper, we prove a version of the Bogomolov--Gieseker inequality for DM surfaces and movable classes of curves:

\begin{teoalpha}[Bogomolov--Gieseker for DM stacks]\label{T:BGmain}
Let $\mathcal X$ be a smooth proper integral Deligne--Mumford stack over $\mathbb C$ of dimension $2$ with projective coarse moduli space, let  $\alpha\in \operatorname{N}_1(\mathcal X)_{\mathbb R}$ be a non-zero movable curve class,  and let $\mathcal E$ be a vector bundle of rank $r$ on $\mathcal X$. If $\mathcal E$ is $\alpha$-semi-stable, then
$$
\Delta(\mathcal E):=2r\cdot c_2(\mathcal E)-(r-1)\cdot c_1(\mathcal E)^2\ge 0.
$$
\end{teoalpha}

This generalizes the classical result of Bogmolov--Gieseker  \cite{G79bogomolov} for smooth projective surfaces and ample curve classes (see also  \cite[Thm.~3.1.4 and 7.3.1]{huybrechts_lehn_2010} and  \cite[Cor.~4.7]{miyaoka87}), as well as the recent result   \cite[Thm.~5.1]{GKP16},   which handles the case of  smooth projective surfaces and movable curve classes.  In the case of \Cref{T:BGmain} where the movable curve class $\alpha$ is obtained as the pull back of an ample curve class on the coarse moduli space, the result follows from 
\cite[Thm.~1.1 and p.32]{JK24BGS}.   
\Cref{T:BGmain} is essentially an immediate consequence of our \Cref{P:StabOnV},  that slope semi-stability for DM stacks is equivalent to slope semi-stability after pull back by a finite cover to a smooth projective variety, which allows one to reduce to the case of smooth projective surfaces (i.e., to \cite[Thm.~5.1]{GKP16}).  
We give the proof in \Cref{S:BG}. 
The theorem is stated only for locally-free sheaves, rather than for torsion-free sheaves, in order to avoid the topic of the Chern classes $c_2(-)$ for torsion-free sheaves on stacks.  By considering a more restricted class of movable curves, we prove a version of \Cref{T:BGmain} for higher dimensional stacks in \Cref{T:BGn}.  
In \Cref{S:MY}, we also consider 
Miyaoka--Yau type inequalities.  Our results in that direction are fairly limited, essentially only applying to the case $\mathcal X=[V/G]$, where $V$ is a smooth projective variety and $G$ is a finite group acting on $G$.  These cases were already known in the literature.

Finally, we note that the primary application we had in mind, motivating our  development of  $\alpha$-stability for stacks here, is an application to foliations and birational geometry of moduli stacks, which will appear in forthcoming work, generalizing results of \cite{CPFol19} for the case where $\mathcal X=[V/G]$.
The present article, as well as the forthcoming work just mentioned, are  the second  and third in a series with the aim of generalizing results of \cite{PS17, WW23} on Viehweg hyperbolicity  to the case of Deligne--Mumford stacks.

\subsection*{Acknowledgements}
The first named author thanks Mihnea Popa for  conversations on the topic, which led to this project.  The first named author also thanks to Jonathan Wise and David Rydh for helpful conversations about the geometry of stacks. The second named author thanks his advisor, Mircea Musta\c{t}\u{a}, for useful discussions and all the support provided, including many
helpful comments on an earlier draft of this paper. 
  The authors are also grateful to the organizers of the Simons Collaborations on Moduli of Varieties Workshop at the University of Utah in November 2024, where their  work on this project began.

\section{Preliminaries}

\subsection{Terminology}

We work over $\mathbb C$.  
A \emph{variety} is an integral separated scheme of finite type over $\mathbb C$. 
An \emph{alteration} $X'\to X$ is a surjective projective generically
finite   morphism of schemes over $\mathbb C$.  
We use the definition of a \emph{Deligne--Mumford (DM) stack} in \cite[Def.~4.1]{LMB}. Note that this differs from the definition in \cite{stacks-project} in that there is the additional hypothesis in \cite[Def.~4.1]{LMB} that the diagonal be representable, separated, and quasi-compact.  We direct the reader to \cite[App.~B]{CMW18} for a discussion of the relationship among various definitions of DM stacks in the literature (see in particular \cite[Fig.~1]{CMW18}).  We emphasize that, with the definition of DM stack that we are using, a morphism from a scheme to a DM stack is schematic (representable by schemes); see e.g., \cite[Lem.~B.20 and Lem.~B.12]{CMW18}.

\subsection{Structure of DM stacks}\label{S:DM-intro}

  The general set-up will be a smooth proper (resp.~separated) integral DM stack $\mathcal X$ of finite type over $\mathbb C$ with coarse moduli space $\pi: \mathcal X\to X$, with the added assumption that the algebraic space $X$ be a projective (resp.~quasi-projective) variety. 
Recall that such a stack admits a finite flat 
morphism $q:V\to \mathcal X$ from a smooth projective
scheme $V$  (\cite[Thm.~1]{KV04} and \cite[Thm.~4.4]{kresch09}, see also \cite[\href{https://stacks.math.columbia.edu/tag/03B6}{\S 03B6}]{stacks-project}); note that $q$ is schematic and projective.  
In this situation we have that $X$ is normal, $\mathbb Q$-factorial, with at worst klt singularities.  The morphism $\pi:\mathcal X\to X$ is flat over the smooth locus of $X$; flatness is an \'etale local property, and so it suffices to consider the case $\mathcal X=[U/G]$ for some smooth variety $U$ and a finite group $G$.  Then, from say \cite[Cor.~14.12]{GW20}, it suffices to show that $U\to U/G$ is flat over the smooth locus of the quotient, which follows from the miracle of flatness \cite[Thm.~23.1, p.179]{matsumura}.

For brevity, we will say that such a stack $\mathcal X$ is a global DM quotient stack if there is a smooth projective (resp.~quasi-projective) variety $V$ over $\mathbb C$ and a finite algebraic group $G$ over $\mathbb C$ acting on $V$ such that $\mathcal X\cong [V/G]$; note that this implies that $X=V/G$.  For context, recall that  $\mathcal X$ is a global quotient stack  if and only if there exists a smooth projective (resp.~quasi-projective) variety  $V'$ and a finite \emph{\'etale} morphism $q':V'\to \mathcal X$ (see the \emph{proof} of \cite[Thm.~(6.1)]{LMB}).

\subsection{Line bundles}

We also recall \cite[Lem. 2.1.2]{AGV08} that for any line bundle $\mathcal L$ on $\mathcal X$, there is a natural number $e$ such that $\mathcal L^{\otimes e}$ is the pull-back via $\pi$ of a line bundle $M$ on $X$.
We also note that for any Zariski open substack $i:\mathcal U\hookrightarrow  \mathcal X$ with complement of codimension at least two, the pull-back morphism
\begin{equation}\label{E:PicXU}
i^*:\operatorname{Pic}(\mathcal X)\to \operatorname{Pic}(\mathcal U)
\end{equation}
is an isomorphism.  Indeed, any line bundle on $\mathcal U$ extends uniquely to a line bundle on $\mathcal X$, since $\mathcal X$ is smooth (and so in particular, $S_2$).   For instance, considering the line bundle as a line bundle on an associated groupoid \cite[\href{https://stacks.math.columbia.edu/tag/03LH}{\S 03LH}]{stacks-project}, the result follows from the fact that on a smooth scheme,  line bundles and morphisms of line bundles extend uniquely over codimension-$2$ loci.

\subsection{Determinants of torsion-free coherent sheaves}
For a torsion-free coherent sheaf $\mathcal F$ on $\mathcal X$ we can define a determinant:
\begin{equation}\label{E:det-def}
\det \mathcal F :=\left(\bigwedge ^{\operatorname{rk}\mathcal F} \mathcal F\right)^{\vee \vee}.
\end{equation}
A key observation is that $\mathcal F$ is locally free over a Zariski open substack $i:\mathcal U\hookrightarrow \mathcal X$ of codimension at least two.  Indeed, this is a local question, and so we can reduce to the question of torsion-free sheaves on smooth varieties.  Considering local rings of codimension $1$, one can see that $\mathcal F$ is locally free in codimension $1$.  Consequently, if $(i^*)^{-1}:\operatorname{Pic}(\mathcal U)\to \operatorname{Pic}(\mathcal X)$ is the inverse of \eqref{E:PicXU}, then 
\begin{equation}\label{E:detVBU}
\det \mathcal F = (i^*)^{-1}(\bigwedge^{\operatorname{rk} \mathcal F} \mathcal F|_{\mathcal U});
\end{equation}
 i.e., the determinant is the unique extension of the determinant of the vector bundle $\mathcal F|_{\mathcal U}$ over $\mathcal U$ to a line bundle on all of $\mathcal X$.

If $\mathcal X=X$ is a smooth projective variety, then this definition of $\det$ clearly agrees with the one in \cite{KMdet}.   Moreover, if $q:V\to \mathcal X$ is a finite flat morphism from a smooth projective variety $V$, then $\det (q^*\mathcal F)=q^*\det (\mathcal F)$.

We can use \eqref{E:detVBU} to show that if 
$$
0\to \mathcal F'\to \mathcal F\to \mathcal F''\to 0
$$
is a short exact sequence of torsion-free coherent sheaves on $\mathcal X$, then 
$$
\det \mathcal F \cong \det \mathcal F' \otimes \det \mathcal F''.
$$
Indeed, there is a common Zariski open substack $i:\mathcal U\hookrightarrow \mathcal X$ with complement of codimension at least two such that all of the sheaves are locally free.  Then one can use that the result holds clearly for short exact sequences of vector bundles, together with \eqref{E:detVBU}.

\subsection{Determinants of coherent sheaves}\label{S:DetCoh}
Let $\mathcal G$ be a  coherent sheaf  on $\mathcal X$
that admits a surjection $\mathcal E\to \mathcal G\to 0$ from a torsion-free coherent sheaf $\mathcal E$.  Note that if $\mathcal X$ has generically trivial stabilizers, then such a surjection exists for any coherent sheaf $\mathcal G$ on $\mathcal X$ \cite[Thm.~1.2]{totaro_2004}, and one may even take $\mathcal E$ to be a vector bundle. 
For  a coherent sheaf $\mathcal G$ that admits a surjection from a torsion-free coherent sheaf $\mathcal E$, there is a short exact sequence
$$
0\to \mathcal F\to \mathcal E\to \mathcal G\to 0
$$
with $\mathcal F$ a torsion-free coherent sheaf.  From this we can define
\begin{equation}\label{E:detGG}
\det (\mathcal G):=\det (\mathcal E)\otimes (\det (\mathcal F))^{-1},
\end{equation}
which, \emph{a priori}, depends on the chosen short exact sequence.  We explain now that the definition is independent of the choice of short exact sequence.  Indeed, in the usual way, given another short exact sequence $0\to \mathcal F'\to \mathcal E'\to \mathcal G\to 0$ with $\mathcal E'$ a torsion-free coherent sheaf, we obtain a third, 
$$
\xymatrix@R=1em{
0 \ar[r]& \mathcal F' \ar[r]& \mathcal E'\ar[r]& \mathcal G\ar[r]& 0 \\
0 \ar[r]& \mathcal F'' \ar[r] \ar[u] \ar[d]& \mathcal E \oplus \mathcal E'\ar[r] \ar[u]\ar[d] & \mathcal G\ar[r] \ar@{=}[u] \ar@{=}[d]& 0 \\
0 \ar[r]& \mathcal F \ar[r]& \mathcal E\ar[r]& \mathcal G\ar[r]& 0 \\
}
$$
where the vertical arrows are quasi-isomorphisms of complexes.  Looking at the open substack $\mathcal U\subseteq \mathcal X$ with complement of codimension at least $2$, where $\mathcal E$, $\mathcal E'$, $\mathcal F$, $\mathcal F'$, and $\mathcal F''$   are all locally free, one arrives in the usual way that $\det (\mathcal E)\otimes (\det (\mathcal F))^{-1}\cong \det (\mathcal E\oplus \mathcal E')\otimes (\det (\mathcal F''))^{-1}\cong \det (\mathcal E')\otimes (\det (\mathcal F'))^{-1}$ (e.g., \cite{KMdet}).

If $\mathcal X=X$ is a smooth projective variety, then this definition of the determinant  clearly agrees with the one in \cite{KMdet}.   
In addition, if 
$$
0\to \mathcal G'\to \mathcal G\to \mathcal G''\to 0
$$
is a short exact sequence of coherent sheaves on $\mathcal X$, all of which   admit surjections from torsion-free coherent sheaves, then 
\begin{equation}\label{E:detLES}
\det \mathcal G \cong \det \mathcal G' \otimes \det \mathcal G''.
\end{equation}
Indeed, one can replace each sheaf by a two-step quasi-isomorphic complex of torsion-free coherent sheaves,  then restrict to the locus where the sheaves are locally free, and use properties of vector bundles and wedge products.

In addition, if $\mathcal G$ is locally free of rank $r$ on an open substack $\mathcal U\subseteq \mathcal X$, then

\begin{equation}\label{E:detUU}
(\det \mathcal G)|_{\mathcal U}= \bigwedge^r(\mathcal G|_{\mathcal U});
\end{equation}
the only thing to observe here is that if $0\to \mathcal F\to \mathcal E\to \mathcal G\to 0$ is a short exact sequence with $\mathcal E$ a torsion-free sheaf, then 
 there is an open substack $\mathcal U'\subseteq \mathcal U$ with complement of codimension $\ge 2$ such that $\mathcal F$ and $\mathcal E$ are also locally free on $\mathcal U'$.
The assertion $(\det \mathcal G)|_{\mathcal U'}= \bigwedge^r(\mathcal G|_{\mathcal U'})$ follows from the definition, and properties of wedge products. The equality then holds on $\mathcal U$, as well.
We will want to use the following assertion for torsion sheaves:

\begin{lem}\label{L:detTors}
If $\mathcal G$ is a torsion coherent sheaf on $\mathcal X$ that admits a surjection from a torsion-free coherent sheaf, then $\det \mathcal G$ is an  effective line bundle; i.e., admits a non-zero global section. 
\end{lem}

\begin{proof}
Let $0 \to \mathcal F\to \mathcal E\to \mathcal G\to 0$ be a short exact sequence with $\mathcal E$ a torsion-free coherent sheaf.  We want to show that $\det (\mathcal E)\otimes (\det(\mathcal F))^{-1}$ is effective.  It suffices to pull-back by an \'etale cover $p:U\to \mathcal X$, and show that the pull-back $p^*(\det (\mathcal E)\otimes (\det(\mathcal F))^{-1})$ is the line bundle associated to an effective divisor $D$ on $U$, with descent data.  The argument in the case of varieties, to show that the determinant of a torsion sheaf is effective,  is to consider codimension $1$ points, in which case one is reduced to modules over a PID; from this it is clear that on a variety $U$, the determinant of a torsion sheaf $\mathcal Q$ is, as a Cartier divisor,  the sum over the codimension $1$ points weighted by the length of $\mathcal Q$  at each codimension $1$ point.  One then extends the associated line bundle (uniquely) over codimension $2$ loci.  As this construction  is functorial,   
the pull-back $p^*(\det (\mathcal E)\otimes (\det(\mathcal F))^{-1})$ is the line bundle associated to an effective divisor $D$ on $U$ with descent data, completing the proof. 
\end{proof}

\subsubsection{First Chern classes of coherent sheaves}\label{S:ChernClXX}
We use the definition of the Chow groups of $\mathcal X$ and intersection theory from \cite{vistoli89}; we denote the Chow groups by $\operatorname{CH}_i(\mathcal X)$ when we refer to dimension and $\operatorname{CH}^i(\mathcal X)$ when we refer to codimension.  When considering intersections, we work in the rational Chow group.   From  \cite{vistoli89} it seems to be well known that one has the notion of Chern classes of vector bundles on $\mathcal X$.
However, as there does not seem to be a standard reference for this, and we only need first Chern classes of line bundles, we only sketch the  special case of line bundles here.  

Following the presentation in \cite[\S 2.5]{fulton}, in the context of \cite{vistoli89}, one finds that for a line bundle $\mathcal L$ on $\mathcal X$, one has for $i\ge 1$ a well-defined homomorphism
\begin{equation}\label{E:c1LL}
c_1(\mathcal L)\cap (-):\operatorname{CH}_i(\mathcal X)\to \operatorname{CH}_{i-1}(\mathcal X).
\end{equation}
In fact, one obtains the full statement of  \cite[Prop.~2.5]{fulton} in the context of line bundles on $\mathcal X$ and classes in the Chow group of $\mathcal X$. i.e., commutativity, the projection formula, etc.

Note also that after making the canonical identification  $\operatorname{CH}_0(\operatorname{Spec}\mathbb C)=\mathbb Z$, we denote by $$\deg:\operatorname{CH}_0(\mathcal X)\to \mathbb Z$$ the push-forward map to $\operatorname{Spec}\mathbb C$.  Then, for a line bundle $\mathcal L$ on $\mathcal X$ and a curve class $\alpha\in \operatorname{CH}_1(\mathcal X)$, we use the notation
\begin{equation}\label{E:c1dot}
c_1(\mathcal L)\cdot \alpha := \deg (c_1(\mathcal L)\cap \alpha).
\end{equation}

For a coherent sheaf $\mathcal F$ on $\mathcal X$ that admits a surjection from a torsion-free coherent sheaf, we define
\begin{equation}\label{E:c1FF}
c_1(\mathcal F):=c_1(\det \mathcal F).
\end{equation}

\begin{rem}
We will sometimes, for brevity, write $c_1(\mathcal L)$ for the cycle class $c_1(\mathcal L)\cap [\mathcal X]$ in $\operatorname{CH}^1(\mathcal X)$.
\end{rem}

\subsection{Log resolution of singularities for DM stacks}

For lack of a convenient reference for the full statement we would like to use, we include the following well-known result (see e.g., \cite[Rem.~3.1]{KTbir23}, as well as the recent preprint \cite{AbSQTW}) regarding resolution of singularities for DM stacks:

\begin{teo}[Hironaka's embedded resolution of singularities]\label{T:Hironaka}
Let $\mathcal X$ be an integral separated DM stack of finite type over $\mathbb C$, and let $\mathcal D\subseteq \mathcal X$ be an effective Cartier divisor on $\mathcal X$.  Then there exists a schematic projective morphism 
$$
\mu:\mathcal X'\to \mathcal X
$$ 
such that:
\begin{enumerate}[label=(\alph*)]
\item \label{E:Hir-1}$\mathcal X'$ is smooth, the exceptional locus  $\mathcal E$ of $\mu$ is a  divisor, and  $\mu^{*}\mathcal D+ \mathcal E$ is a simple normal crossing divisor.

\item \label{E:Hir-2} $\mu$ is an isomorphism over $\mathcal X-(\operatorname{Sing}\mathcal X\cup \operatorname{Sing}(\mathcal D))$. 

\end{enumerate}
Moreover,  if $\mathcal X$ is normal, the natural morphism $\mathcal O_{\mathcal X}\to \mu_*\mathcal O_{\mathcal X'}$ is an isomorphism, and if $\mathcal X$ has a (quasi-) projective coarse moduli space, then so does $\mathcal X'$.
%%%%%%%%%%%%%%%%%%%%%%%%%%%%%%%%%%%%%%%%%%%%%%%%%%%%%%%%%%%
\ifArxiv
\else 
\qed
\fi
%%%%%%%%%%%%%%%%%%%%%%%%%%%%%%%%%%%%%%%%%%%%%%%%%%%%%%%%%%%
\end{teo}

\begin{rem}
A similar embedded resolution of singularities holds for closed substacks of higher codimension by first blowing up the stack along the closed substack to obtain a divisor, and then employing \Cref{T:Hironaka}; to keep the statement of the theorem cleaner, we have only stated the theorem for divisors. 
\end{rem}

%%%%%%%%%%%%%%%%%%%%%%%%%%%%%%%%%%%%%%%%%%%%%%%%%%%%%%%%%%%
\ifArxiv
While the \Cref{T:Hironaka} seems to be well-known, since we were not aware of a reference in the literature stating the result as we have above, we have included a proof in \Cref{S:ResSing}.
\fi
%%%%%%%%%%%%%%%%%%%%%%%%%%%%%%%%%%%%%%%%%%%%%%%%%%%%%%%%%%%

\subsection{Resolution to locally free sheaves} Here we explain how a result of Rossi implies that after pull backs to a blow-up, one can replace coherent sheaves on stacks by locally free sheaves: 

\begin{teo}[{\cite[Thm.~3.5]{Rossi68}}]\label{T:Rossi}
Let $\mathcal X$ be a proper integral DM stack over $\mathbb C$ with projective coarse moduli space, and let $\mathcal E$ be a coherent sheaf on $\mathcal X$.  There exists a blow-up $\mu:\mathcal X'\to \mathcal X$ from a smooth proper integral DM stack over $\mathbb C$ with projective coarse moduli space  such that $(\mu^*\mathcal E)^{\vee \vee}$ is locally free.
\end{teo}

\begin{proof}
The proof in 
\cite[Thm.~3.5]{Rossi68}
is done for irreducible analytic spaces.  The proof, essentially  as written,  works in  the algebraic category.   
The proof caries over directly to algebraic stacks as the proof is done first locally, constructing an ideal sheaf to blow-up locally; then it is shown that these all agree on overlaps; adapting this argument to the \'etale topology gives the proof for stacks. 
\end{proof}

\subsection{Galois covers and quotient stacks} Suppose that $q:V\to X$ is a finite flat morphism of smooth projective varieties.  Then it is 
well-known (taking integral closures in extensions of fraction fields, for instance) that there is a normal projective variety $V'$ with a finite morphism $V'\to V$, and a finite group $G$ acting on $V'$ so that the composition $V'\to V\to X$ exhibits $X$ as the quotient $X=V'/G$.  

On the other hand, 
suppose that $q:V\to \mathcal X$ is a finite flat morphism from a smooth projective variety $V$ to a smooth proper integral DM stack $\mathcal X$ over $\mathbb C$ with projective coarse moduli space.  If $q$ is not \'etale, there does not exist a finite cover $V'\to V$ from a projective variety (or even proper DM stack) $V'$ with a finite group $G$ acting on $V'$ so that the composition $V'\to V\to \mathcal X$ exhibits $\mathcal X$ as $[V'/G]$, since  the \'etale morphism $V'\to [V'/G]\cong \mathcal X$ could not factor through the ramified morphism $q$.  
The following version of \cite[Thm.~C]{rydhComact} gives us a way around this issue.  The statement includes the notion of a stacky blow-up, which is a composition of usual blow-ups with root stacks, as well as the notion of an admissible stacky blow-up, which means the blow-up is an isomorphism over the specified locus; we direct the reader to \cite[\S 8, Def.~(9.3), Def.~(3.2)]{rydhComact} for more details.  

\begin{teo}[{\cite[Thm.~C]{rydhComact}}]\label{T:Rydh}
Let $q:V\to \mathcal X$ be a finite flat morphism from a smooth projective variety $V$ to a smooth proper integral DM stack $\mathcal X$ over $\mathbb C$ with projective coarse moduli space.  Let $\mathcal U\subseteq \mathcal X$ be the open substack over which $q$ is \'etale.  There exists a stacky $\mathcal U$-admissible
blow-up $\sigma:{\mathcal X}'\to \mathcal X$ from a smooth proper DM stack $\mathcal X'$ over $\mathbb C$ with projective coarse moduli space, and  a generically finite morphism $\mathcal V'\to V$ from a smooth proper integral  DM stack $\mathcal V'$ over $\mathbb C$ with projective coarse moduli space, which is \'etale over $q^{-1}(\mathcal U)$, and a finite group $G$ acting on $\mathcal V'$ so that there is a commutative diagram
\begin{equation}\label{E:TRydh-dgm}
\xymatrix{
\mathcal V'\ar[r] \ar[d]^{/G}& V \ar[d]^q\\
[\mathcal V'/G]\cong {\mathcal X'}\ar[r]^<>(0.5)\sigma& \mathcal X
}
\end{equation}
\end{teo}

\begin{proof}
Applying \cite[Thm.~C]{rydhComact} as stated, one obtains a diagram
$$
\xymatrix{
\widetilde {\mathcal V}\ar[r] \ar[d]^{\text{\'et}}& V \ar[d]^q\\
 \widetilde {\mathcal X}\ar[r]^<>(0.5)\mu& \mathcal X
}
$$
where $\mu$ is a $\mathcal U$-admissible stacky blow-up (see \cite[Def.~(9.3)]{rydhComact}),  
the morphism $\widetilde {\mathcal V}\to \widetilde {\mathcal X}\times_{\mathcal X}V$ is a $q^{-1}(\mathcal U)$-admissible blow-up   (see \cite[\S 4]{rydhComact}), 
and $\widetilde {\mathcal V}\to \widetilde {\mathcal X}$ is \'etale.  Taking a strong resolution of singularities, one can assume that $\widetilde {\mathcal X}$ is smooth (and therefore that $\widetilde {\mathcal V}$ is smooth).  The key take-aways are the following.  
The morphism $\mu$ is an isomorphism over 
$\mathcal U$, and  $\widetilde {\mathcal V}\to V$ is an isomorphism over 
$q^{-1}(\mathcal U)$.

Now, given the \'etale morphism $\widetilde {\mathcal V}\to \widetilde {\mathcal X}$, the usual Galois category type arguments show that one then gets a proper integral DM stack $\mathcal V'$ with projective coarse moduli space and a morphism   $\mathcal V'\to \widetilde {\mathcal V}$ so that the composition $\mathcal V'\to \widetilde {\mathcal V}\to \widetilde {\mathcal X}$ has  Galois group $G$ over $\widetilde {\mathcal X}$, so that $\widetilde {\mathcal X}\cong [\mathcal V'/G]$.  This is explained in the proof of \cite[Thm.~(6.1)]{LMB} using fibered products, and a similar argument is given in the proof of \cite[Lem.~(13.4)]{rydhComact}.  Alternatively, one can use the notion of a Galois category, as in SGA1.~Exp.~V, Sec.~4, Axioms G1--G6  (see also \cite[\href{https://stacks.math.columbia.edu/tag/0BMQ}{\S  0BMQ}]{stacks-project}).  Note that, as the morphism $\mathcal V'\to [\mathcal V'/G]$ is \'etale and the quotient $[\mathcal V'/G]\cong \widetilde {\mathcal X}$ is smooth, we have that $\mathcal V'$ is smooth.
\end{proof}

\begin{rem} Let $\mathcal X$ be a normal DM stack. For a morphism $\mu:\widetilde {\mathcal X}\to \mathcal X$ that is  a composition of blow-ups and root stacks,  the natural morphism $\mathcal O_{\mathcal X}\to \mu_*\mathcal O_{\widetilde {\mathcal X}}$ is an isomorphism.  Indeed, for blow-ups, if the base is normal, the push-forward of the structure sheaf is the structure sheaf. Given a root stack $\rho:\widetilde {\mathcal X}\to \mathcal X$, we also  have that the natural morphism $\mathcal O_{\mathcal X}\to \rho_*\mathcal O_{\widetilde {\mathcal X}}$ is an isomorphism. Indeed, this is a local question, 
and so just comes down to the general statement about quotient stacks by finite group, as the root stack ${\widetilde {\mathcal X}}$ is locally modeled on $\left[\left(\underline{\operatorname{Spec}}_{\mathcal X}\mathcal O_{\mathcal X} [t]/(t^r-s)\right)/\mu_r\right]\to \mathcal X$.  

\end{rem}

\begin{rem}\label{R:VBdescentRS}
For a root stack $\rho:\widetilde {\mathcal X}\to \mathcal X$, we note that given a  vector bundle $\widetilde {\mathcal E}$ on $\widetilde {\mathcal X}$, if the action of the stabilizer on each fiber of $\widetilde {\mathcal E}$ is trivial, then the natural morphism $\pi^*\pi_*\widetilde {\mathcal E}\to \widetilde {\mathcal E}$ is an isomorphism.  
 This follows from standard arguments about vector bundles on quotients of DM stacks  (see e.g., the argument of \cite[Thm.~3.1.1(1)]{chuck_07} regarding line bundles on root stacks, or the argument of \cite[Lem.~2.1.2]{AGV08} regarding line bundles on coarse moduli spaces of DM stacks, or  the argument of \cite[Thm.~10.3]{alper13} regarding vector bundles on stacks with good moduli spaces). 
\end{rem}

\section{Movable curve classes on stacks}\label{S:mov-class-XX}

Here we define the notion of  movable classes of curves on stacks.  Our presentation closely follows the presentation in 
\cite{GKP16} for projective varieties.
Let $\mathcal X$ be a normal proper $\mathbb Q$-factorial DM stack over $\mathbb C$ with projective coarse moduli space $\pi:\mathcal X\to X$.
Under these hypotheses,  $X$ is also normal, projective, and $\mathbb Q$-factorial.
In this section we are weakening the hypothesis on $\mathcal X$ (elsewhere we always assume that $\mathcal X$ is smooth) simply for convenience, so that we can address the stack $\mathcal X$ and the coarse moduli space $X$ simultaneously in the discussion.  The reader who wishes may assume that $\mathcal X$ is smooth, and then repeat the proofs for the coarse moduli space $X$.    We use the Chow groups and intersection theory developed for stacks in \cite{vistoli89}; 
recall that for all $i$ we have that the push-forward morphism  $\pi_*:\operatorname{CH}_i(\mathcal X)_{\mathbb Q}\stackrel{\sim}{\to}\operatorname{CH}_i(X)_{\mathbb Q}$ is an isomorphism \cite[Prop.~(6.1(i))]{vistoli89}.

The starting point for our discussion is the non-degenerate bilinear intersection pairing
\begin{equation}\label{E:BilIntPairXX}
\xymatrix{
\operatorname{N}^1(\mathcal X)_{\mathbb R}\times \operatorname{N}_1(\mathcal X)_{\mathbb R} \ar[r]^<>(0.5){(-)\cdot(-)} \ar[d]_{\pi_*\times \pi_*}^\cong&  \mathbb R \ar@{=}[d]\\\operatorname{N}^1( X)_{\mathbb R}\times \operatorname{N}_1(X)_{\mathbb R} \ar[r]^<>(0.5){(-)\cdot(-)}&  \mathbb R.
}
\end{equation}
To be precise, and to connect this with our discussion of first Chern classes on DM stacks, we recall a definition of this intersection product.  We identify $\operatorname{CH}^1(\mathcal X)_{\mathbb Q}=\operatorname{Pic}^1(\mathcal X)_{\mathbb Q}$ via the $\mathbb Q$-factorial hypothesis.  Then for a Cartier divisor $\mathcal D$ on $\mathcal X$ and a class $\alpha\in 
\operatorname{CH}_1(\mathcal X)$, we define $\mathcal D\cdot \alpha$ as $\deg (c_1(\mathcal O_V(\mathcal D))\cap \alpha)$, where degree represents the push-forward to a point (and we have identified the Chow group of the point with $\mathbb Z$).  We then extend this pairing $\mathbb Q$-linearly in each entry, and therefore get definitions also for Weil divisor classes.  We can then extend $\mathbb R$-linearly to get the map above.  
From the definitions, applying the same to $X$, one can check that the diagram \eqref{E:BilIntPairXX} above is commutative.

\begin{dfn}[Movable class]\label{D:MoveClass}
A class $\alpha\in \operatorname{N}_1(\mathcal X)_{\mathbb R}$ is called \emph{pseudo-movable}, or simply \emph{movable}, if we have  
$\mathcal D\cdot \alpha \ge 0$ for every pseudo-effective divisor class $\mathcal D\in \operatorname{N}^1(\mathcal X)_{\mathbb R}$.  
\end{dfn}

We are using the definition of pseudo-effective divisors from \cite[Def.~2.4]{CMZpositivity}; 
i.e., a divisor class $\mathcal D$ is pseudo-effective on $\mathcal X$ if some positive multiple is a $\mathbb Z$-Cartier divisor that descends to a pseudo-effective $\mathbb Z$-Cartier divisor on the coarse moduli space $X$.
    The pseudo-effective  divisor classes define a cone, equal to the cone obtained by taking the pre-image of the pseudo-effective cone on $X$ under the isomorphism   $\pi_*:\operatorname{CH}_{\dim X-1}(\mathcal X)_{\mathbb R}\stackrel{\sim}{\to}\operatorname{CH}_{\dim X-1}(X)_{\mathbb R}$. 
  The set of movable  classes form a closed convex cone in $ \operatorname{N}_1(\mathcal X)_{\mathbb R}$ that we will call the \emph{(pseudo-)movable cone}; it is, by definition, the dual cone to the pseudo-effective cone.

\begin{rem}
  As the pseudo-effective cone on $X$ is the closure of the cone generated by effective divisors, then using the isomorphism  $\pi_*:\operatorname{CH}_{\dim X-1}(\mathcal X)_{\mathbb R}\stackrel{\sim}{\to}\operatorname{CH}_{\dim X-1} (X)_{\mathbb R}$, one has that the pseudo-effective cone on $\mathcal X$ is the closure of the cone generated by effective divisors on $\mathcal X$.
  Consequently,  a class  $\alpha\in \operatorname{N}_1(\mathcal X)_{\mathbb R}$ is movable on $\mathcal X$ if and only if  $\mathcal D\cdot \alpha \ge 0$ for every effective $\mathbb Z$-divisor $\mathcal D$ on $\mathcal X$. 
Finally, it is clear from the definition of movable classes and \eqref{E:BilIntPairXX} that the isomorphism $\pi_*: \operatorname{N}_1(\mathcal X)_{\mathbb R}\stackrel{\sim}{\to} \operatorname{N}_1(X)_{\mathbb R}$ takes the cone of movable classes on $\mathcal X$ isomorphically to the cone of movable classes on $X$.  
\end{rem}

Given a morphism $q':\mathcal X'\to \mathcal X$ of smooth proper DM stacks over $\mathbb C$ with projective coarse moduli spaces $\pi':\mathcal X'\to X'$ and $\pi:\mathcal X\to X$ respectively, and associated morphism $\psi':X'\to X$, 
the push-forward map of divisors respects $\mathbb R$-linear and numerical equivalence, and therefore, there is a commutative diagram of $\mathbb R$-linear maps
 $$
\xymatrix@C=2em@R=1.5em{
\operatorname{N}^1(\mathcal X')_{\mathbb R}\ar[r]^{q'_*} \ar[d]_{\pi'_*}^\cong& \operatorname{N}^1(\mathcal X)_{\mathbb R} \ar[d]_{\pi_*}^\cong\\
\operatorname{N}^1( X')_{\mathbb R}\ar[r]^{\psi'_*}& \operatorname{N}^1(X)_{\mathbb R}.\\
}$$

Considering duals, and the pairing \eqref{E:BilIntPairXX}, 
we obtain a commutative diagram
\begin{equation}\label{E:num-pullXX}
\xymatrix@R=1em{
&\operatorname{N}^1(X')^\vee_{\mathbb R} \ar@{-}[d]^{\cong} \ar[ld]_{\pi'^\vee_*}^\cong&& \operatorname{N}^1(X)^\vee_{\mathbb R} \ar@{->}[ll] _{\psi'^\vee_*} \ar[ld]_{\pi^\vee_*}^\cong\\
\operatorname{N}^1(\mathcal X')^\vee_{\mathbb R} \ar[dd]^{\cong}&\ar[d]& \operatorname{N}^1(\mathcal X)^\vee_{\mathbb R} \ar@{->}[ll] _<>(0.25){q'^\vee_*}&\\
&\operatorname{N}_1( X')_{\mathbb R}\ar@{<--}[rr]^<>(0.75){\psi'^*} && \operatorname{N}_1(X)_{\mathbb R} \ar[uu]^{\cong}\\
\operatorname{N}_1(\mathcal X')_{\mathbb R}\ar@{<--}[rr] ^{q'^*} \ar[ru]^<>(0.6){\pi'_*}_<>(0.8)\cong && \operatorname{N}_1(\mathcal X)_{\mathbb R} \ar[uu]^<>(0.25){\cong} \ar[ru]_{\pi_*}^\cong&\\
}
\end{equation}

where $q'^*$ and $\psi'^*$  are defined to make the diagram commute.  The definition of $q'^*$ can be interpreted via the projection formula: 
\begin{align}\label{E:projformnumpull}
\mathcal D'\cdot q'^*(\alpha) &= q'_*\mathcal D'\cdot \alpha \ \ \ \text{for all $\alpha\in \operatorname{N}_1(\mathcal X)_{\mathbb R}$ and all $\mathcal D'\in \operatorname{N}^1(\mathcal X')_{\mathbb R}$.}
\end{align}
We will call $q'^*$ 
 the \emph{numerical pull-back}.

\begin{rem}\label{R:Fpb=Npb}
Note that if $q'$  is flat, then the numerical pull-back agrees with the flat pull-back; indeed, the flat pull-back satisfies the same projection formula \eqref{E:projformnumpull}.    
In addition, from the definition, it is clear that the numerical pull-back is functorial.
\end{rem}

\begin{rem}\label{R:EdotMov0}
We observe that $\mathcal D'\cdot q'^*\alpha=0$ for any $\mathcal D'$ such that $q'_*\mathcal D'=0$. In particular, if $\mathcal D'\subseteq \mathcal X'$ is an irreducible effective divisor contracted by $q'$,  then $\mathcal D'\cdot q'^*\alpha=0$.  The same statement holds for divisors and curve classes on $X'$. 
\end{rem}

\begin{lem}\label{L:mov-pullXX} 
Let $q':\mathcal X'\to \mathcal X$ be a morphism of irreducible normal proper $\mathbb Q$-factorial  DM stacks over $\mathbb C$ with projective coarse moduli spaces.
\begin{enumerate}[label=(\arabic*)]

\item \label{L:mov-pullXX1} If $\alpha'\in \operatorname{N}_1(\mathcal X')_{\mathbb R}$ is movable, then $q'_*\alpha'\in \operatorname{N}_1(\mathcal X)_{\mathbb R}$ is  movable.

\item \label{L:mov-pullXX2}  If   $\alpha\in \operatorname{N}_1(\mathcal X)_{\mathbb R}$ is movable, then  $q'^*\alpha\in \operatorname{N}_1(\mathcal X')_{\mathbb R}$ is  movable. If $q'$ is generically finite and surjective, then the converse holds.

\item \label{L:mov-pullXX3}  If $q'$ is finite and flat, and  $\alpha\in \operatorname{N}_1(\mathcal X)_{\mathbb R}$ is movable, then there is a movable class $\alpha'\in \operatorname{N}_1( \mathcal X')_{\mathbb R}$  such that $\alpha = q'_*\alpha'$.

\end{enumerate}
\end{lem}

\begin{proof}
Considering coarse moduli spaces, this all reduces to the case of morphisms of normal projective $\mathbb Q$-factorial varieties, which is well-known.  The proofs are identical in the case of stacks; for convenience, we include the proofs here.

  \ref{L:mov-pullXX1}
Suppose that $\alpha'$ is movable.  Then for every effective divisor $\mathcal D$ on $\mathcal X$ we have ${\mathcal D}\cdot q'_*\alpha' = q'^*{\mathcal D}\cdot \alpha'\ge 0$.

\ref{L:mov-pullXX2} Suppose that $\alpha$ is movable.  Then for every effective divisor ${\mathcal D}'$ on ${\mathcal X}'$ we have ${\mathcal D}'\cdot q'^*\alpha = q'_*{\mathcal D}'\cdot \alpha \ge 0$.  Conversely, suppose that  $q'^*\alpha$ is movable.  Given an effective divisor ${\mathcal D}$ on ${\mathcal X}$, by projection formula, since $q'$ is generically finite and surjective, $\mathcal D =  \lambda q'_* q'^*\mathcal{D}$ for some positive rational number $\lambda$. So, for every effective divisor ${\mathcal D}$ on ${\mathcal X}$  we have ${\mathcal D}\cdot \alpha = \lambda q'_*q'^*{\mathcal D}\cdot \alpha = \lambda q'^*{\mathcal D}\cdot q'^*\alpha \ge 0$.  

\ref{L:mov-pullXX3} Suppose that $\alpha$ is movable.  Then from \ref{L:mov-pullXX2} we have that $q'^*\alpha$ is movable.   Then, with the assumption that $q'$ is finite and flat, we have that $q'_*(q'^*\alpha)$ is a positive multiple of $\alpha$.  
\end{proof}

\section{$\alpha$-slope stability on DM stacks}
 The proof of \cite[Thm.~7.6]{CPFol19} given in \cite[Thm.~1]{Schnell17Epi}, which we aim to extend to DM stacks,    depends on some results from \cite[\S 2]{GKP16} regarding stability with respect to movable classes of curves.  In this subsection, we generalize these results to DM stacks.  
Throughout this section we work on a smooth proper integral DM stack $\mathcal X$ over $\mathbb C$ with projective coarse moduli space.

\subsection{Preliminaries on slope}\label{S:SlopeStab}
Recall that the definition of a movable class of curves $\alpha$ on $\mathcal X$ is given in \S\ref{S:mov-class-XX}.  
As with any $1$-cycle, we can define the associated \emph{$\alpha$-slope} of a torsion-free coherent sheaf $\mathcal E$ of positive rank on $\mathcal X$ as
\begin{equation}\label{E:Defmmualpha}
\mu_\alpha(\mathcal E):=\frac{c_1(\mathcal E)\cdot \alpha}{\operatorname{rk}(\mathcal E)}.
\end{equation}
We say that $\mathcal E$ is \emph{$\alpha$-semi-stable} if $\mu_\alpha(\mathcal F)\le \mu_\alpha(\mathcal E)$ for any coherent subsheaf $0\subsetneq \mathcal F\subseteq \mathcal E$.  We say that  $\mathcal E$ is \emph{$\alpha$-stable} if $\mu_\alpha(\mathcal F)<\mu_\alpha(\mathcal E)$ for an coherent subsheaf $0\subsetneq \mathcal F\subsetneq \mathcal E$ \emph{such that} $\operatorname{rk}(\mathcal F)<\operatorname{rk}(\mathcal E)$.  

The following notion is convenient for working with $\alpha$-semi-stability:  
  For a  torsion-free coherent sheaf $\mathcal E$ of positive rank on $\mathcal X$, we define
$$
\mu_\alpha^{\max}(\mathcal E):=\sup \{\mu_\alpha(\mathcal F) : 0\ne \mathcal F\subseteq \mathcal E \text{ a coherent subsheaf }\}.
$$
  Clearly $\m{E}$ is $\alpha$-semistable if and only if $\mu_{\alpha}^{\max}(\m{E})=\mu_{\alpha}(\mathcal E)$.

Working with slopes, we will frequently use the following elementary arithmetic regarding real numbers $a_1,a_2,b_1,b_2$, with $b_1,b_2>0$:
\begin{equation}\label{E:SlpArith}
\operatorname{min}\{\frac{a_1}{b_1},\frac{a_2}{b_2}\}\leq \frac{a_1+a_2}{b_1+b_2}\leq \operatorname{max} \{\frac{a_1}{b_1},\frac{a_2}{b_2}\}
\end{equation}
and if equality holds on one side, then equality holds on both sides.

\begin{lem}\label{L:SlpIncSlpInc}
    If $0\ra \m{F} \ra \m{E} \ra \m{Q} \ra 0$ is a short exact sequence of non-zero torsion-free coherent sheaves, then the following are equivalent: 
    \begin{align*}
        \mu_{\alpha}(\m{F})\leq \mu_{\alpha}(\m{E}) \\
        \mu_{\alpha}(\m{E})\leq \mu_{\alpha}(\m{Q}).
    \end{align*}
   Moreover,  one is an equality if and only if the other one is, too. 
\end{lem}

\begin{proof}
    Use that  $c_1(\m{E})=c_1(\m{F})+c_1(\m{Q})$ and $\operatorname{rk}(\m{E})=\operatorname{rk}(\m{F})+\operatorname{rk}(\mathcal Q)$, and then apply \eqref{E:SlpArith}.    
\end{proof}

Using \Cref{L:SlpIncSlpInc} we see that 
 $\mathcal E$ is \emph{$\alpha$-semi-stable} if and only if  $\mu_\alpha(\mathcal E)\le \mu_\alpha(\mathcal Q)$ for any torsion-free quotient $\mathcal E\to \mathcal Q\to 0$.
 We similarly define 
    $$\mu_{\alpha}^{\min} (\m{E}):=\inf \{\mu_{\alpha}(\m{Q}) :   \m{E}\ra \m{Q}\ra 0 \,\, \textrm{torsion-free quotient}\},$$
 and we see that $\mathcal E$ is $\alpha$-semi-stable if and only if $\mu_\alpha^{min}(\mathcal E)=\mu_\alpha(\mathcal E)$.  
 
Summarizing the discussion above, we give a generalization of \cite[Pro.~2.16]{GKP16}:

\begin{pro}

\label{P:GKP16-2.16}
  If  $\gamma:\mathcal E\to \mathcal E''$ is a nonzero morphism of  torsion-free coherent sheaves of positive rank,  and $\mathcal E$ is  $\alpha$-semistable, then $\mu_{\alpha}(\mathcal E)\le \mu_\alpha(\operatorname{Im}\gamma)$.
\end{pro}

\begin{proof}
This follows immediately from the discussion above.
\end{proof}

With the above, we can give a generalization of \cite[Cor.~2.17]{GKP16}:

\begin{cor}

\label{C:GKP16-2.17}
Suppose that $\mathcal E$ and $\mathcal E''$ are non-zero torsion-free coherent sheaves on $\mathcal X$.  

\begin{enumerate} [label=(\arabic*)]
\item \label{E:C:GKP16-2.17(1)}  If $\mathcal E$ and $\mathcal E''$ are $\alpha$-semistable and if $\mu_\alpha(\mathcal E)>\mu_\alpha(\mathcal E'')$, then 
$\operatorname{Hom}_{\mathcal O_{\mathcal X}}(\mathcal E,\mathcal E'')=0$.

\item  If $\mathcal E$ and $\mathcal E''$ are stable and if $\mu_\alpha(\mathcal E)=\mu_\alpha(\mathcal E'')$, then any nonzero morphism $\mathcal E\to \mathcal E''$ is injective and generically an isomorphism. 
Given such a morphism $\mathcal E\to \mathcal E''$, if we assume in addition
that $\mathcal E$ is saturated in $\mathcal E''$, then the morphism $\mathcal E\to \mathcal E''$  is an isomorphism. (Recall that $\mathcal E$ is said to be saturated in $\mathcal E''$ if the quotient $\mathcal E'' /\mathcal E$ is
torsion-free.) 

\end{enumerate}
\end{cor}

%%%%%%%%%%%%%%%%%%%%%%%%%%%%%%%%%%%%%%%%%%%%%%%%%%%%%%%%%%%%%
\ifArxiv
\begin{proof}
   The corollary  follows from \Cref{P:GKP16-2.16} and  elementary Chern class computations. We show the first part. Let $\gamma:\mathcal E\to \mathcal E''$ be a nonzero homomorphism of $\mu_{\alpha}$-semi-stable sheaves with $\mu_{\alpha}(\m{E})> \mu_{\alpha}(\m{E}'')$.    Then we have $\mu_\alpha(\mathcal E)\le \mu_\alpha(\operatorname{Im}\gamma)\le \mu_\alpha(\mathcal E'')$, where the first inequality is from \Cref{P:GKP16-2.16}, and the second comes from the $\alpha$-semi-stability of $\mathcal E''$.  This contradicts $\mu_{\alpha}(\m{E})> \mu_{\alpha}(\m{E}'')$.
\end{proof}
\else
\begin{proof}
The proof is the same as in the case of varieties (e.g., \cite[Cor.~2.17]{GKP16}).
\end{proof}
\fi
%%%%%%%%%%%%%%%%%%%%%%%%%%%%%%%%%%%%%%%%%%%%%%%%%%%%%%%%%%%%%

We also have the following observation, generalizing \cite[Lem.~2.5]{CPFol19}.

\begin{lem}
\label{L:CP2.5}
Suppose that $\mathcal E$ and $\mathcal E''$ are non-zero torsion-free coherent sheaves on $\mathcal X$.  
\begin{enumerate} [label=(\alph*)]
\item \label{E:L:CP2.5a} If $\mu_\alpha^{\min}(\mathcal E)>\mu_\alpha^{\max}(\mathcal E'')$, then $\operatorname{Hom}(\mathcal E,\mathcal E'')=0$.

\item  \label{E:L:CP2.5b} If $\mu_\alpha^{\max}(\mathcal E'')<0$, then $H^0(\mathcal X,\mathcal E'')=0$.  
\end{enumerate}
\end{lem}

%%%%%%%%%%%%%%%%%%%%%%%%%%%%%%%%%%%%%%%%%%%%%%%%%%%%%%%%%%%%%
\ifArxiv
\begin{proof}
It suffices to prove \ref{E:L:CP2.5a}, since $\mu_\alpha^{\min}(\mathcal O_{\mathcal X})=0$ (any surjection $\mathcal O_{\mathcal X}\to \mathcal Q\to 0$ to a torsion-free $\mathcal Q$ must have $\mathcal Q$ of rank-$1$, and therefore torsion kernel, and so must be an isomorphism).  For \ref{E:L:CP2.5a}, we observe that a morphism $\mathcal E\to \mathcal E''$ factors as $\mathcal E\twoheadrightarrow \mathcal Q\hookrightarrow \mathcal E''$, for some torsion-free coherent sheaf $\mathcal Q$ (since the image of the morphism $\mathcal E\to \mathcal E''$ is contained in the torsion-free sheaf $\mathcal E''$).  We then have $\mu_\alpha^{\min}(\mathcal E)\le \mu_\alpha(\mathcal Q)\le \mu_\alpha^{\max}(\mathcal E'')$, contradicting our assumption that $\mu_\alpha^{\min}(\mathcal E)>\mu_\alpha^{\max}(\mathcal E'')$.
\end{proof}
\else
\begin{proof}
The proof is the same as in the case of varieties (e.g., \cite[Lem.~2.5]{CPFol19}).
\end{proof}
\fi
%%%%%%%%%%%%%%%%%%%%%%%%%%%%%%%%%%%%%%%%%%%%%%%%%%%%%%%%%%%%%

\subsection{Maximal destabilizing subsheaves and Harder--Narasimhan filtrations}
Note that so far we have not used anything special about the class $\alpha$.  We now discuss maximally destabilizing subsheaves and the  Harder--Narhisimhan (HN) filtration, where we use that $\alpha$ is a movable class.  More precisely, we use that  for a  movable class  $\alpha$ and a torsion sheaf $\mathcal G$ that admits a surjection from a torsion-free sheaf, one has that $c_1(\mathcal G)\cdot \alpha\ge 0$; indeed, this follows from the definition of a movable class, since the determinant of such a torsion coherent sheaf is effective (\Cref{L:detTors}). The presentation in this section is taken essentially verbatim from \cite[\S 1.3]{huybrechts_lehn_2010}, with some minor clarifying details, replacing their slope stability by $\mu_{\alpha}$-stability, and replacing their projective variety with $\mathcal X$.

The main observation is the following:

\begin{lem}\label{L:muSat}
Let $\m{E}$ be a non-zero torsion-free coherent 
sheaf and let $\mathcal F \subseteq \mathcal  E$  be a non-zero subsheaf, with saturation in $\mathcal E$ denoted $\mathcal F^{\operatorname{sat}}$.  Then $\mu_\alpha(\mathcal F)\le \mu_\alpha(\mathcal F^{\operatorname{sat}})$.  
\end{lem}

\begin{proof}
The key point is that $\alpha$ is a movable class.  We consider the short exact sequence
$$
0\to \mathcal F\to \mathcal F^{\operatorname{sat}}\to\mathcal F^{\operatorname{sat}}/\mathcal F\to 0.
$$
We have $c_1(\mathcal F^{\operatorname{sat}})= c_1(\mathcal F)+c_1(\mathcal F^{\operatorname{sat}}/\mathcal F)$, where here we are using that by construction, $\mathcal F^{\operatorname{sat}}/\mathcal F$ is a torsion sheaf that is the quotient of a torsion-free sheaf, and therefore has a determinant as defined in \S \ref{S:DetCoh}.    Since $\alpha$ is a movable class, and $\mathcal F^{\operatorname{sat}}/\mathcal F$ is torsion, we have $c_1(\mathcal F^{\operatorname{sat}}/\mathcal F)\cdot \alpha \ge 0$ (\Cref{L:detTors}), so that $c_1(\mathcal F)\cdot \alpha \le c_1(\mathcal F^{\operatorname{sat}})\cdot \alpha$.  As both $\mathcal F$ and $\mathcal F^{\operatorname{sat}}$ have the same rank, we find that $\mu_\alpha(\mathcal F)\le \mu_\alpha(\mathcal F^{\operatorname{sat}})$.   
\end{proof}

\begin{cor}\label{C:st-ss}
If $\mathcal E$ is $\alpha$-stable, then it is $\alpha$-semi-stable.
\end{cor}

%%%%%%%%%%%%%%%%%%%%%%%%%%%%%%%%%%%%%%%%%%%%%%%%% 
\ifArxiv
\begin{proof}
Suppose that $\mathcal E$ were $\alpha$-stable, but not $\alpha$-semi-stable.  Then, from the definition of $\alpha$-semi-stable,  there would be a non-zero $\mathcal F\subseteq \mathcal E$ such that $\mu_\alpha(\mathcal F)>\mu_\alpha(\mathcal E)$.
By \Cref{L:muSat}, one may assume that $\mathcal F$ is saturated. 
  On the other hand, by the definition of $\alpha$-stable, one would have to have $\operatorname{rk}(\mathcal F)=\operatorname{rk}(\mathcal E)$, so that $\mathcal F=\mathcal E$. But this would contradict $\mu_\alpha(\mathcal F)>\mu_\alpha(\mathcal E)$. 
\end{proof}
\else
\begin{proof}
By virtue of  \Cref{L:muSat}, the argument is the same as in the case of varieties (e.g., \cite[Cor.~2.14]{GKP16}).
\end{proof}
\fi 
%%%%%%%%%%%%%%%%%%%%%%%%%%%%%%%%%%%%%%%%%%%%%%%%%%%%%%%%%%

The main result we will want is the following, establishing the existence of  maximally destabilizing subsheaves, generalizing \cite[Cor.~2.24]{GKP16}: 

\begin{pro}[{Maximally destabilizing subsheaf}]

\label{P:GKP16-2.24}
Let $\m{E}$ be a non-zero torsion-free coherent sheaf. There exists a sheaf $\mathcal F \subseteq \mathcal  E$  such that the following holds.
\begin{enumerate}[label=(\alph*)]

\item \label{I:P:GKP16-2.24a} $\mu_\alpha(\mathcal F)=\mu_\alpha^{\max}(\mathcal E)$. 

\item  \label{I:P:GKP16-2.24b} If $\mathcal F'\subseteq \mathcal E$ is any other sheaf such that $\mu_\alpha(\mathcal F')=\mu_\alpha^{\max}(\mathcal E)$, then $\mathcal F'\subseteq \mathcal F$. 

\end{enumerate}
\end{pro}

The sheaf $\mathcal F$ is called the \emph{maximally destabilizing subsheaf of $\mathcal E$}. It is clearly unique, $\alpha$-semi-stable, and saturated in $\mathcal E$.  
 Note the somewhat confusing terminology, since if $\mathcal E$ is $\alpha$-semi-stable, then it still has a maximally destabilizing subsheaf, namely itself, but nevertheless, $\mathcal E$ is $\alpha$-semi-stable.  In general, the maximally destabilizing subsheaf  $\mathcal F$, being saturated in $\mathcal E$, also satisfies the condition that if $\mathcal F\ne \mathcal E$ (i.e., if $\mathcal E$ is not $\alpha$-semi-stable), then  $\operatorname{rk}(\mathcal F)<\operatorname{rk}(\mathcal E)$.

%%%%%%%%%%%%%%%%%%%%%%%%%%%%%%%%%%%%%%%%%%%%%%%%% 
\ifArxiv
\begin{proof}
We start by showing \ref{I:P:GKP16-2.24a}, following the proof of \cite[Prop.~2.22]{GKP16}.  Assume for the sake of contradiction that $\mu_\alpha(\mathcal F)<\mu_\alpha^{\max}(\mathcal E)$ for any nonzero coherent subsheaf $\mathcal F\subseteq \mathcal E$.  One can then find a sequence of subsheaves $(\mathcal F_i)_{i\in \mathbb N}$ such that 
\begin{itemize}
\item $\mu_\alpha(\mathcal F_i)<\mu_\alpha(\mathcal F_{i+1})$.
\item $\mu_\alpha^{\max}(\mathcal E)-\frac{1}{i} < \mu_\alpha(\mathcal F_i)<\mu_\alpha^{\max}(\mathcal E)$.
\item The sheaves $\mathcal F_i$ are saturated in $\mathcal E$ (use \Cref{L:muSat}), and all have the same rank $r$.
\end{itemize}
In addition, we can assume that the rank $r$ is maximal among all sequences of sheaves satisfying the conditions above.  We will arrive at a  contradiction by showing that:
\begin{quote}
Given any real number $\epsilon>0$, there exists a non-zero subsheaf $\mathcal G_\epsilon \subseteq \mathcal E$ such that $\mu_\alpha(\mathcal G_\epsilon)> \mu_\alpha^{\max}(\mathcal E)-\epsilon$ and such that $\operatorname{rk}(\mathcal G_\epsilon)>r$. 
\end{quote}
To this end, let $\epsilon>0$ be a real number.  Then let $i$ be a natural number such that $\frac{1}{i}<\frac{\epsilon}{2}$, let $j>i$ be a larger natural number, and observe that $\mu_\alpha(\mathcal F_i)<\mu_{\alpha}(\mathcal F_{j})$, so that clearly $\mathcal F_i\ne \mathcal F_{j}$.  Consequently, as both $\mathcal F_i$ and $\mathcal F_{j}$ are saturated in $\mathcal E$, and of the same rank, they cannot be contained in one another (see also \cite[Lem.~A.10]{GKP14}), the sum $\mathcal G_\epsilon:=\mathcal F_i+\mathcal F_j$ therefore satisfies $\operatorname{rk}\mathcal G_\epsilon>r$ (look on an open set where $\mathcal F_i$ and $\mathcal F_j$ are locally free).   An elementary Chern class computation, spelled out in \cite[Lem.~A.12]{GKP14}, and which we reproduce below for completeness, shows that 
\begin{equation}\label{E:GKP14muEst}
\mu_\alpha(\mathcal G_\epsilon)>\mu_\alpha^{\max}(\mathcal E)-\frac{1}{i}-\frac{1}{j}>\mu_\alpha^{\max}(\mathcal E)-\epsilon,
\end{equation}
 giving our contradiction.

We now reproduce the Chern class computation for the left hand side of  \eqref{E:GKP14muEst}.
We consider the short exact sequence
$$
0\to \mathcal F_i\cap \mathcal F_i \to \mathcal F_i\oplus \mathcal F_j \to \mathcal F_i+\mathcal F_j\to 0.
$$
Since all the sheaves involved are torsion-free, hence in particular locally free in codimension 1, it follows that
\begin{align}
\label{E:detFiFjGKP} \det (\mathcal F_i+\mathcal F_j)&= \det \mathcal F_i+\det \mathcal F_j -\det (\mathcal F_i\cap \mathcal F_j)\\
\label{E:rkFiFjGKP}
\operatorname{rk}(\mathcal F_i+\mathcal F_j)&=\underbrace{\operatorname{rk}\mathcal F_i+\operatorname{rk}\mathcal F_j}_{=2r} - \operatorname{rk}(\mathcal F_i\cap \mathcal F_j).
\end{align}
We then have
\begin{align*}
& \operatorname{rk}(\mathcal F_i+\mathcal F_j) \cdot \mu_\alpha(\mathcal F_i+\mathcal F_j)\\
& = r\cdot \mu_\alpha(\mathcal F_i)+r\cdot \mu_\alpha(\mathcal F_j)-\operatorname{rk}(\mathcal F_i\cap \mathcal F_j)\cdot \mu_\alpha (\mathcal F_i\cap \mathcal F_j) & \text{Def.~of } \ \mu_\alpha \text { and } \eqref{E:detFiFjGKP} \\
&\ge r\cdot \left( \mu_\alpha(\mathcal F_i)+\mu_\alpha (\mathcal F_j)\right)-\operatorname{rk}(\mathcal F_i\cap \mathcal F_j)\cdot \mu_\alpha^{\max}(\mathcal E) & \text{Def.~of } \ \mu_\alpha^{\max}\\
&> r\cdot \left(2\mu_\alpha^{\max}(\mathcal E)-\frac{1}{i}-\frac{1}{j}\right)-\operatorname{rk}(\mathcal F_i\cap \mathcal F_j)\cdot \mu_\alpha^{\max}(\mathcal E) & \text{Def.~of } \ \mathcal F_i\\
& = -r\left(\frac{1}{i}+\frac{1}{j} \right)+\operatorname{rk}(\mathcal F_i+\mathcal F_j)\cdot \mu_\alpha^{\max}(\mathcal E). & \text{by \eqref{E:rkFiFjGKP}}
\end{align*}
Dividing by $\operatorname{rk}(\mathcal F_i+\mathcal F_j)>r$ gives the left hand side of \eqref{E:GKP14muEst}, completing the proof of  \ref{I:P:GKP16-2.24a}.

We now show \ref{I:P:GKP16-2.24b}.  Using part \ref{I:P:GKP16-2.24a}, we have a saturated subsheaf $\mathcal F\subseteq \mathcal E$ of maximal $\alpha$-slope, and we can choose such an $\mathcal F$ with   maximal rank.  Let $\mathcal F'$ be any other subsheaf of maximal slope, and assume for the sake of contradiction that  $\mathcal F'$ is not contained in $\mathcal F$.  Then we have $\mathcal F\subseteq \mathcal F+\mathcal F'$, with $\operatorname{rk}(\mathcal F)<\operatorname{rk}(\mathcal F+\mathcal F'$).  Thus, we must have $\mu_\alpha(\mathcal F+\mathcal F')<\mu_\alpha(\mathcal F)$.   We now do the following computation: 

\begin{align*}
&\operatorname{rk}(\mathcal F+\mathcal F')\cdot \left(\underbrace{\mu_\alpha(\mathcal F+\mathcal F') -\mu_\alpha(\mathcal F)}_{(*)}\right)+\left(\operatorname{rk}\mathcal F'-\operatorname{rk}(\mathcal F\cap \mathcal F')\right)\cdot \left(\underbrace{\mu_\alpha(\mathcal F)-\mu_\alpha(\mathcal F')}_{(**)}\right)\\
&= c_1(\mathcal F+\mathcal F')\cdot \alpha -\left(\frac{\operatorname{rk}(\mathcal F+\mathcal F')}{\operatorname{rk}(\mathcal F)}\right)\cdot c_1(\mathcal F)\cdot \alpha\\
&+\left(\frac{\operatorname{rk}(\mathcal F')}{\operatorname{rk}(\mathcal F)}\right)\cdot c_1(\mathcal F)\cdot \alpha -  c_1(\mathcal F')\cdot \alpha -\left(\frac{\operatorname{rk}(\mathcal F\cap \mathcal F')}{\operatorname{rk}(\mathcal F)}\right)\cdot c_1(\mathcal F)\cdot \alpha  + \left(\frac{\operatorname{rk}(\mathcal F\cap \mathcal F')}{\operatorname{rk}(\mathcal F')}\right)\cdot c_1(\mathcal F')\cdot \alpha\\
\end{align*}
\begin{align*}
&= c_1(\mathcal F+\mathcal F')\cdot \alpha + \left(-\left(\frac{\operatorname{rk}(\mathcal F+\mathcal F')}{\operatorname{rk}(\mathcal F)}\right)+\left(\frac{\operatorname{rk}(\mathcal F')}{\operatorname{rk}(\mathcal F)}\right) - \left(\frac{\operatorname{rk}(\mathcal F\cap \mathcal F')}{\operatorname{rk}(\mathcal F)}\right) \right) c_1(\mathcal F)\cdot \alpha\\
&+\left(-1+\left(\frac{\operatorname{rk}(\mathcal F\cap \mathcal F')}{\operatorname{rk}(\mathcal F')}\right)\right) c_1(\mathcal F')\cdot \alpha\\
\end{align*}
\begin{align*}
&= c_1(\mathcal F+\mathcal F')\cdot \alpha - c_1(\mathcal F)\cdot \alpha 
+\left(-1+\left(\frac{\operatorname{rk}(\mathcal F\cap \mathcal F')}{\operatorname{rk}(\mathcal F')}\right)\right) c_1(\mathcal F')\cdot \alpha \\
&= -c_1(\mathcal F\cap \mathcal F')\cdot \alpha +  \left(\frac{\operatorname{rk}(\mathcal F\cap \mathcal F')}{\operatorname{rk}(\mathcal F')}\right) c_1(\mathcal F')\cdot \alpha\\
&= \operatorname{rk}(\mathcal F\cap \mathcal F')\cdot \left(\underbrace{\mu_\alpha (\mathcal F')-\mu_\alpha(\mathcal F\cap \mathcal F')}_{(***)}\right)
\end{align*}
Now, since $\mu_\alpha(\mathcal F+\mathcal F')<\mu_\alpha(\mathcal F)=\mu_\alpha(\mathcal F')$ ($=\mu_\alpha^{\max}(\mathcal E)$), we have that $(*)$ is negative and $(**)$ is zero.  Thus $(***)$ must be negative, so that $\mu_\alpha(\mathcal F')<\mu_\alpha(\mathcal F\cap \mathcal F')$, which is a contradiction, as $\mu_\alpha(\mathcal F')$ was assumed to be maximal. 
\end{proof}
\else 
\begin{proof}
The proof is the same as in the case of varieties (see, e.g., \cite[Prop.~2.22 and Cor.~2.24]{GKP16}).
\end{proof}
\fi 
%%%%%%%%%%%%%%%%%%%%%%%%%%%%%%%%%%%%%%%%%%%%%%%%%%%%%%%%%%

\begin{dfn}[Harder--Narisimhan (HN) filtration]
    Fix a movable class $\alpha$. Let $\m{E}$ be a non-zero torsion-free coherent sheaf on $\X$. A \emph{Harder--Narasimhan (HN) filtration} for $\m{E}$ is an increasing filtration 
    \[0=HN_0(\m{E})\subseteq HN_1(\m{E})\subseteq \cdots \subseteq HN_\ell(\m{E})=\m{E},\]
    so that the factors $\operatorname{gr}_i^{HN}(\m{E}):= HN_i(\m{E})/HN_{i-1}(\m{E})$ for $i=1,\cdots,\ell$ are $\mu_{\alpha}$-semistable torsion-free sheaves with slopes $\mu_i:=\mu_\alpha(\operatorname{gr}_i^{HN}(\m{E}))$ satisfying: 
    
    \begin{equation}\label{E:HNslopes}
  \mu_{\alpha}^{\max'}(\m{E}):=\mu_1>\cdots >\mu_\ell=:\mu_{\alpha}^{\min '}(\m{E}).
\end{equation}

\end{dfn}

\begin{lem}[Existence of HN filtrations]\label{L:HNFil}
For any torsion-free coherent sheaf $\mathcal E$, there exists a unique HN filtration for $\mathcal E$.  

\end{lem}

%%%%%%%%%%%%%%%%%%%%%%%%%%%%%%%%%%%%%%%%%%%%%%%%% 
\ifArxiv
\begin{proof}
    Let $\m{E}$ be a torsion-free coherent sheaf. We first show existence. Let $\m{E}_1$ be the maximal destabilizing subsheaf from \Cref{P:GKP16-2.24}. By induction on the rank, as the quotient is torsion-free, we can assume that $\m{E}/\m{E}_1$ has a HN filtration $0=\m{G}_0\subseteq \cdots \subseteq \m{G}_{\ell-1}=\m{E}/\m{E}_1$. If $\m{E}_{i}\subseteq \m{E}_{i+1}$
    denotes the pre-image of $\m{G}_i$, all that is left to show is 
     that the slopes are decreasing.  Inductively, it suffices to show that $\mu_{\alpha}(\m{E}_1)>\mu_{\alpha}(\m{E}_2/\m{E}_1)$. For this consider the short exact sequence $0\to \mathcal E_1\to \mathcal E_2\to \mathcal E_2/\mathcal E_1\to 0$, and employ \Cref{L:SlpIncSlpInc} together with the maximality of $\mu_\alpha(\mathcal E_1)$. 

    Now we argue uniqueness. Assume that $\m{E}_{\bullet}$ and $\m{E}_{\bullet}'$ are two HN filtrations. Without loss of generality, $\mu_{\alpha}(\m{E}_1')\geq \mu_{\alpha}(\m{E}_1)$. Let $j$ be minimal so that $\m{E}_1'\subseteq \m{E}_j$. Then, the composition 
    $$\m{E}_1'\ra \m{E}_j\ra \m{E}_j/\m{E}_{j-1}$$
    is a non-trivial morphism of semistable sheaves, and, by \Cref{C:GKP16-2.17}\ref{E:C:GKP16-2.17(1)}, this implies the first inequality below:
    $$\mu_{\alpha}(\m{E}_j/\m{E}_{j-1})\geq \mu_{\alpha}(\m{E}_1')\geq \mu_{\alpha}(\m{E}_1)\geq \mu_{\alpha}(\m{E}_j/\m{E}_{j-1}).$$
    The second inequality above is by assumption, and the third is by definition of the slopes in an HN filtration. 
    Hence, equality holds everywhere, and the last equality then implies $j=1$. So, $\m{E}_1'\subseteq \m{E}_1$. 
    At the same time, the fact that $\mu_{\alpha}(\m{E}_1')= \mu_{\alpha}(\m{E}_1)$
means we may reverse the roles of $\m{E}_{\bullet}$ and $\m{E}_{\bullet}'$ in the argument above, and so we get equality $\m{E}_1=\m{E}_1'$. By induction on rank, we can assume that $\m{E}/\m{E}_1$ has a unique HN filtration. This shows that $\m{E}_i/\m{E}_1=\m{E}_i'/\m{E}_1$, and so  we are done.
\end{proof}
\else
\begin{proof}
The proof is the same as in the case of varieties (see, e.g., \cite[\S 1.3]{huybrechts_lehn_2010}).
\end{proof}
\fi
%%%%%%%%%%%%%%%%%%%%%%%%%%%%%%%%%%%%%%%%%%%%%%%%% 

\begin{lem}\label{L:min-slope-equals-negative-slope-of-dual-max}
   For a torsion-free coherent sheaf $\mathcal E$ on $\X$, we have:
\begin{enumerate} [label=(\arabic*)]
\item \label{L:min-slope-equals-negative-slope-of-dual-max1}
   $\mu_{\alpha}^{\min}(\mathcal E)=-\mu_{\alpha}^{\max}(\mathcal E^\vee)$. 
   
\item \label{L:min-slope-equals-negative-slope-of-dual-max2}
 The supremum in $\mu_{\alpha}^{\max}(\mathcal E)$ is a maximum and 
the infemum in $\mu_{\alpha}^{\min}(\mathcal E)$ is a minimum. 

\item \label{L:min-slope-equals-negative-slope-of-dual-max3} In the notation of \eqref{E:HNslopes}, regarding the slopes in the HN filtration, 
 $\mu_{\alpha}^{\max}(\mathcal E)=\mu_{\alpha}^{\max '}(\mathcal E)$  and $\mu_{\alpha}^{\min}(\mathcal E)=\mu_{\alpha}^{\min'}(\mathcal E)$.
\end{enumerate}
   
\end{lem}

%%%%%%%%%%%%%%%%%%%%%%%%%%%%%%%%%%%%%%%%%%%%%%%%% 
\ifArxiv
\begin{proof} We have already seen the first part of  \ref{L:min-slope-equals-negative-slope-of-dual-max2},  i.e.,  that the 
 supremum in $\mu_{\alpha}^{\max}$ is a maximum, achieved by the maximally destabilizing subsheaf, and, since the maximally destabilizing subsheaf was used to construct the first step in the HN filtration, we also have the first part of \ref{L:min-slope-equals-negative-slope-of-dual-max3}, as well, that   $\mu_{\alpha}^{\max}(\mathcal E)=\mu_{\alpha}^{\max '}(\mathcal E)$.

We next move to prove \ref{L:min-slope-equals-negative-slope-of-dual-max1}.  
We start by establishing the related result, that $\mu_{\alpha}^{\min}(\mathcal E)= \mu_{\alpha}^{\min} (\mathcal E^{\vee \vee})$. 
  To begin the argument, given a surjection $\mathcal E\xrightarrow{e} \mathcal {Q}\ra 0$ to a torsion-free coherent sheaf $\mathcal Q$, we get $\m{E}^{\vee \vee}\xrightarrow{e^{\vee \vee}}\mathcal Q^{\vee\vee}$ which is not necessarily surjective but gives $\mathcal E^{\vee \vee}\twoheadrightarrow \operatorname{Im}(e^{\vee \vee })\subseteq \mathcal Q^{\vee \vee}$, where the first map is a surjection. Moreover, $\operatorname{coker}(\operatorname{Im}(e^{\vee \vee})\hookrightarrow \mathcal Q^{\vee \vee})$ is torsion and admits a surjection from the torsion-free sheaf $\mathcal Q^{\vee \vee}$. 
  Hence, as $\alpha$ is a movable class and as $\mathcal Q$ is locally free away from codimension $2$ and so satisfies $\det \mathcal Q=\det \mathcal Q^{\vee \vee}$, by \Cref{L:detTors} we have $\mu_{\alpha}(\mathcal Q)=\mu_{\alpha}(\mathcal Q^{\vee \vee})\geq \mu_{\alpha}(\operatorname{Im}(e^{\vee \vee}))$. So, $\mu_{\alpha}^{\min}(\mathcal E)\geq \mu_{\alpha}^{\min} (\mathcal E^{\vee \vee})$. On the other hand, for any surjection $\mathcal E^{\vee \vee}\ra \mathcal Q \ra 0 $, we obtain, via the inclusion $\mathcal E\ra \mathcal E^{\vee \vee}$, a generic surjection $\mathcal E\to \mathcal Q$. Call the image $\mathcal Q_0:=\operatorname{Im}(\mathcal E\ra \mathcal Q)$. Then, $\operatorname{coker}(\mathcal Q_0\hookrightarrow  \mathcal Q)$ is torsion and admits a surjection from the torsion-free sheaf $\mathcal Q$; hence we get  a quotient $\mathcal E\to \mathcal Q_0\to 0$ such that $\mu_{\alpha}(\mathcal Q_0)\leq \mu_{\alpha}(\mathcal Q)$. 
  Hence 
  $\mu_{\alpha}^{\min}(\mathcal E)\leq \mu_{\alpha}^{\min} (\mathcal E^{\vee \vee})$, completing the proof that $\mu_{\alpha}^{\min}(\mathcal E)= \mu_{\alpha}^{\min} (\mathcal E^{\vee \vee})$. 
    
  Now, we show that  $\mu_{\alpha}^{\min}(\mathcal E)=-\mu_{\alpha}^{\max}(\mathcal E^\vee)$. We start   by observing that for a torsion-free sheaf $\mathcal E$, one has $c_1(\m{E}^\vee)=-c_1(\mathcal E)$; restrict to an open with complement of codimension $2$ over which $\mathcal E$ is locally free. 
   One has:
    \begin{align*}
       \mu_{\alpha}^{\min}(\mathcal E) & := \inf\{\mu_{\alpha}(\mathcal Q)\, \mid \, \forall \, \mathcal E \ra \mathcal Q \ra 0, \, \textrm{torsion-free quotient} \}\\
        & = -\sup\{\mu_{\alpha}(\mathcal Q^\vee)\, \mid \, \forall \, \mathcal E \ra \mathcal Q \ra 0, \, \textrm{torsion-free quotient} \}\\
        & = -\sup\{\mu_{\alpha}(\mathcal Q^\vee)\, \mid \, \forall \, \mathcal Q^\vee \subseteq \mathcal E^\vee  \, \textrm{where $\mathcal Q$ is a torsion free quotient of $\mathcal E$} \}\\
        &\geq -\mu_{\alpha}^{\max}(\mathcal E^\vee).
    \end{align*}
Hence, we have obtained that $\mu_{\alpha}^{\min}(\mathcal E) \geq -\mu_{\alpha}^{\max}(\mathcal E^\vee)$.  Because $\mu_{\alpha}^{\min}(\mathcal E)=\mu_{\alpha}^{\min}(\mathcal E^{\vee \vee})$, to show the inequality above is equality, it is sufficient to show $\mu_{\alpha}^{\min}(\mathcal E^\vee) \leq -\mu_{\alpha}^{\max}(\mathcal E)$. To see this, note that
\begin{align*}
    -\mu_{\alpha}^{\max}(\mathcal E) & :=-\sup\{\mu_{\alpha}(\mathcal F)\,\mid\, \mathcal F\subseteq \mathcal E\} \\ 
    & = \inf\{\mu_{\alpha}(\mathcal F^\vee)\,\mid\, \mathcal F\subseteq \mathcal E\} \\
    & \geq \inf\{\mu_{\alpha}(\mathcal Q)\,\mid\, \mathcal E^\vee \ra \mathcal Q \ra 0, \textrm{ torsion free }\}
\end{align*}
    where the last inequality follows because $e:\mathcal F\hookrightarrow  \mathcal E$ implies that $e^\vee:\mathcal E^\vee \ra \mathcal F^\vee$ is generically surjective, and we get $\mu_{\alpha}(\operatorname{Im}(e^\vee))\leq \mu_{\alpha}(\mathcal F^\vee)$. Hence, we obtain $\mu_{\alpha}^{\min}(\mathcal E^\vee) \leq -\mu_{\alpha}^{\max}(\mathcal E)$.  This completes the proof that $\mu_{\alpha}^{\min}(\mathcal E) = -\mu_{\alpha}^{\max}(\mathcal E^\vee)$.

    Now, we show the second parts of \ref{L:min-slope-equals-negative-slope-of-dual-max2} and \ref{L:min-slope-equals-negative-slope-of-dual-max3}, that the infimum  in $\mu_{\alpha}^{\min}(\mathcal E)$ is actually a minimum, and that it is achieved by the last quotient of the HN filtration ($\mathcal E/HN_{\ell-1}\mathcal E$).  For brevity, we will write $\mathcal E_i:=HN_i\mathcal E$ in what follows.  
 Our first step is to show that 
 \begin{equation}\label{E:min-slope-equals-negative-slope-of-dual-max1}
\mu_\alpha(\mathcal E)\ge \mu_\alpha(\mathcal E/\mathcal E_{\ell-1}).
\end{equation}   
For this, we note that the HN filtration on $\mathcal E$ was constructed by pull-back from the HN filtration on $\mathcal E/\mathcal E_1$, and so applying \eqref{E:min-slope-equals-negative-slope-of-dual-max1} inductively to $\mathcal E/\mathcal E_1$ (where we are inducing on the rank), we may conclude that $\mu_\alpha(\mathcal E/\mathcal E_1)\ge \mu_\alpha(\mathcal E/\mathcal E_{\ell-1})$.    We also have by construction that $\mathcal E_1$ is the maximally destabilizing subsheaf of $\mathcal E$, so that $\mu_\alpha(\mathcal E_1)\ge \mu_\alpha(\mathcal E)$.  Now applying \Cref{L:SlpIncSlpInc} to the short exact sequence  $0\to \mathcal E_1\to \mathcal E\to \mathcal E/\mathcal E_1\to 0$, we may conclude that $\mu_\alpha(\mathcal E)\ge \mu_\alpha(\mathcal E/\mathcal E_1)$.  In conclusion we have $\mu_\alpha(\mathcal E)\ge \mu_\alpha(\mathcal E/\mathcal E_1)\ge \mu_\alpha(\mathcal E/\mathcal E_{\ell-1})$, establishing \eqref{E:min-slope-equals-negative-slope-of-dual-max1}.

    Now consider a torsion-free quotient $\mathcal E\to \mathcal Q\to 0$, which for the sake of contradiction we assume satisfies  $\mu_{\alpha}(\mathcal Q)<\mu_{\alpha}(\mathcal E/\mathcal E_{\ell-1})$.  Applying \eqref{E:min-slope-equals-negative-slope-of-dual-max1} to $\mathcal Q$, and taking the further quotient $\mathcal E\twoheadrightarrow \mathcal Q\twoheadrightarrow \mathcal Q/HN_{\max}\mathcal Q\to 0$, where $HN_{\max}\mathcal Q$ is the last step in the HN filtration for $\mathcal Q$, we may replace $\mathcal Q$ with $\mathcal Q/HN_{\max}\mathcal Q$, and assume that $\mathcal Q$ is $\alpha$-semi-stable. 
     Now, since we are assuming that $\mu_{\alpha}(\mathcal Q)<\mu_{\alpha}(\mathcal E/\mathcal E_{\ell-1})$, then we have $\mu_{\alpha}(\mathcal Q)<\mu_{\alpha}(\mathcal E_1)$. By \Cref{C:GKP16-2.17}\ref{E:C:GKP16-2.17(1)},  we get that $\mathcal E\ra \mathcal Q$ factors through a surjection $\mathcal E/\mathcal E_1\to \mathcal Q$. Applying this procedure repeatedly (i.e., the next step would be to consider the morphism $\mathcal E/\mathcal E_1\to \mathcal Q$, and then to conclude that the morphism is zero on $\mathcal E_2/\mathcal E_1$, so that it factors through a surjection $\mathcal E/\mathcal E_2 \to \mathcal Q$), we obtain that $\mathcal E/\mathcal E_{\ell-1}$ surjects onto $\mathcal Q$; but then again by \Cref{C:GKP16-2.17}\ref{E:C:GKP16-2.17(1)}, there is no such non-zero map, hence a contradiction and we are done.
\end{proof}
\else 
\begin{proof}
The proof is the same as in the case of varieties and is left to the reader.
\end{proof}
\fi 
%%%%%%%%%%%%%%%%%%%%%%%%%%%%%%%%%%%%%%%%%%%%%%%%% 

\begin{cor}\label{C:DualStab}
A coherent sheaf  $\mathcal E$  is $\alpha$-semi-stable if and only if $\mathcal E^{\vee}$ is $\alpha$-semi-stable.
\end{cor}

\begin{proof}
We have $\mathcal E$ is $\alpha$-semi-stable $\iff$  $\mu_\alpha^{\max}(\mathcal E)=\mu_\alpha(\mathcal E)$ $\iff$   $-\mu_\alpha^{\max}(\mathcal E)=-\mu_\alpha(\mathcal E)$ $\iff$  $\mu_\alpha^{\min}(\mathcal E^\vee)=\mu_\alpha(\mathcal E^\vee)$ $\iff$ $\mathcal E ^\vee$ is $\alpha$-semi-stable. 
\end{proof}

\subsection{Jordan--Holder filtrations} 
We now consider Jordan--Holder filtrations; our presentation here closely follows that of \cite[\S 1.5]{huybrechts_lehn_2010}.

\begin{lem}[Existence of a stable destabilising subsheaf]\label{L:stdes}
 If $\mathcal{E}$ is $\alpha$-semi-stable, there exists a saturated $\alpha$-stable subsheaf
  $\mathcal{E}' \subseteq \mathcal{E}$ of slope $\mu_{\alpha}(\mathcal{E}') = \mu_{\alpha}^{\max} (\mathcal{E})=\mu_\alpha(\mathcal E)$.
\end{lem}

%%%%%%%%%%%%%%%%%%%%%%%%%%%%%%%%%%%%%%%%%%%%%%%%% 
\ifArxiv 
\begin{proof}
   If $\mathcal{E}_1:=\mathcal{E}$ is not $\alpha$-stable, there exists
  a sheaf $\mathcal{E}_2 \subsetneq \mathcal{E}_1$ that is also of maximal slope, but of smaller
  rank: $\operatorname{rk} \mathcal{E}_2 < \operatorname{rk}  \mathcal{E}_1$.  From \Cref{L:muSat}, one can assume that $\mathcal E_2$ is saturated.  Iterate this process, in order to
  construct a strictly decreasing sequence of saturated sheaves of maximal slope,
  $\mathcal{E}_1 \supsetneq \mathcal{E}_2 \supsetneq \cdots$.  The process terminates because
  the rank decreases in each step. 
\end{proof}
\else 
\begin{proof}
The proof is the same as in the case of varieties (e.g., \cite[\S 1.5]{huybrechts_lehn_2010}).
\end{proof}
\fi 
%%%%%%%%%%%%%%%%%%%%%%%%%%%%%%%%%%%%%%%%%%%%%%%%% 

\begin{cor}[Existence of Jördan-Hölder filtrations]\label{cor:JH}
  If $\mathcal{E}$ is $\alpha$-semi-stable, 
  there exists a \emph{Jordan--H\"older filtration} of $\mathcal E$, that is, a filtration
  $0 = \mathcal{E}_0 \subsetneq \mathcal{E}_1 \subsetneq \cdots \subsetneq \mathcal{E}_r = \mathcal{E}$ where
  each quotient $\mathcal{Q}_i := {\mathcal{E}_i}/{\mathcal{E}_{i-1}}$ is torsion-free, $\alpha$-stable,
  and with slopes $\mu_{\alpha} \bigl( \mathcal{Q}_i \bigr) = \mu_{\alpha}\bigl( \mathcal{E} \bigr)$. 
\end{cor}

%%%%%%%%%%%%%%%%%%%%%%%%%%%%%%%%%%%%%%%%%%%%%%%%% 
\ifArxiv 
\begin{proof}
Let $\mathcal E_1\subseteq \mathcal E$ be a stable saturated subsheaf with slope $\mu_\alpha(\mathcal E_1)=\mu_\alpha(\mathcal E)$,  from \Cref{L:stdes}. If $\mathcal E_1\ne \mathcal E$, then consider the short exact sequence
$$
0\to \mathcal E_1\to \mathcal E \to \mathcal E/\mathcal E_1\to 0.
$$
As $\mathcal E_1$ is saturated and not equal to $\mathcal E$, it must be that $\operatorname{rk}(\mathcal E_1)<\operatorname{rk}(\mathcal E)$.  
 Using \Cref{L:SlpIncSlpInc}, and the fact that $\mu_\alpha(\mathcal E_1)=\mu_\alpha(\mathcal E)$, it is clear that $\mu_\alpha(\mathcal E/\mathcal E_1)=\mu_\alpha(\mathcal E)$.  We have that $\mathcal E/\mathcal E_1$ is torsion-free, and we claim that it is $\alpha$-semistable.
For this, any nontrivial  quotient  $\mathcal E/\mathcal E_1 \rightarrow \mathcal F''$ is also a non-trivial quotient of $\mathcal{E}$. Hence, $\mu_\alpha(\mathcal E/\mathcal E_1)=\mu_\alpha(\mathcal E)>\mu_\alpha(\mathcal F'')$ since $\mathcal{E}$ is $\alpha$-semistable.

Now, by induction on the rank, there is a Jordan--H\"older filtration for $\mathcal E/\mathcal E_1$.  Pulling back to $\mathcal E$ gives the rest of the Jordan--H\"older filtration for $\mathcal E$. 
\end{proof}
\else 
\begin{proof}
The proof is the same as in the case of varieties (e.g., \cite[\S 1.5]{huybrechts_lehn_2010}).
\end{proof}
\fi 
%%%%%%%%%%%%%%%%%%%%%%%%%%%%%%%%%%%%%%%%%%%%%%%%% 

\begin{rem}[Refined Harder--Narasimhan filtration]\label{rem:RHNF}
  Combining the Harder--Narasimhan filtration with the  
  Jordan--H\"older filtration, one obtains a \emph{refined
  Harder--Narasimhan filtration}
  $$0 = \mathcal{E}_0 \subsetneq \mathcal{E}_1 \subsetneq \cdots \subsetneq \mathcal{E}_r = \mathcal{E}$$ where
  each quotient $\mathcal{Q}_i := {\mathcal{E}_i}/{\mathcal{E}_{i-1}}$ is torsion-free and $\alpha$-stable,
  and where the sequence of slopes $\mu_{\alpha}\bigl( \mathcal{Q}_i \bigr)$ is decreasing
  (though not necessarily strictly decreasing), and the ranks are strictly increasing.  Note also that by construction, $\mu_\alpha(\mathcal E_1)=\mu_\alpha(\mathcal Q_1)=\mu_\alpha^{\max}(\mathcal E)$. 
\end{rem}

\subsection{$\alpha$-semi-stability and tensor products of vector bundles}

The goal of this section is to prove the following result generalizing  \cite[Prop.~5.2]{CPT11}: 

\begin{pro}

\label{P:TensorProdStab}
 Let $\mathcal X$ be a smooth proper DM stack over $\mathbb C$ with projective coarse moduli space, and let $\alpha$ be a movable curve class.  If $\mathcal E$ and $\mathcal E'$ are $\alpha$-semi-stable 
vector bundles on $\mathcal X$, then so is  $\mathcal E\otimes \mathcal E'$.
\end{pro}

The key step in the proof is to show the following, which allows us to reduce to the case of smooth projective varieties:

\begin{pro}\label{P:StabOnV}
 Let $\mathcal X$ be a smooth proper integral DM stack over $\mathbb C$ with projective coarse moduli space, let $q:V\to \mathcal X$ be a finite flat morphism from a smooth projective variety $V$, let $\alpha$ be a movable curve class on  $\mathcal X$, and let $\mathcal E$ be a vector bundle on $\mathcal X$.  Then $q^*\alpha$ is a movable curve class on $V$ (\Cref{L:mov-pullXX}\ref{L:mov-pullXX2}), and $\mathcal E$ is $\alpha$-semi-stable if and only if $q^*\mathcal E$ is $q^*\alpha$-semi-stable.  
\end{pro}

Before we prove \Cref{P:StabOnV}, we explain how it is used to prove \Cref{P:TensorProdStab}.

\begin{proof}[Proof of \Cref{P:TensorProdStab}]
Suppose that $\mathcal E$ and $\mathcal E'$ are $\alpha$-semi-stable 
vector bundles on $\mathcal X$, and 
let $q:V\to \mathcal X$ be a finite flat morphism from a smooth projective variety $V$.  
Then by virtue of \Cref{P:StabOnV}, we have that  $q^*\mathcal E$ and $q^*\mathcal E'$ are $q^*\alpha$-semi-stable on $V$. Consequently, $q^*\mathcal E\otimes q^*\mathcal E'$ is $q^*\alpha$-semi-stable on $V$ by \cite[Thm.~4.2]{GKP16}.  Applying \Cref{P:StabOnV} to $\mathcal E\otimes \mathcal E'$, we see that $\mathcal E \otimes \mathcal E'$ is $\alpha$-semi-stable on $\mathcal X$.   
\end{proof}

\begin{rem}
We take a moment here to discuss the proof of \Cref{P:StabOnV}.  One direction, the proof that an unstable vector bundle pulls-back to an unstable vector bundle is relatively easy; one simply pulls back the destabilizing subsheaf (\Cref{L:GFstab-1}).  The complication is showing the converse,  that if a vector bundle pulls-back to an unstable vector bundle, then one can descend a destabilizing subsheaf, to show the original vector bundle was unstable.  One case where this is easy is for Galois covers \Cref{L:Galstab} (this is essentially equivalent to the argument in \cite[Prop.~2.6]{CPFol19}); however, unlike the case of projective varieties, where finite covers can be extended to Galois covers, in contrast, for DM stacks, one only knows that one  can achieve this after a stacky blow-up (via a result of Rydh \Cref{T:Rydh}; \cite[Thm.~C]{rydhComact}).  The main content of what follows is to explain how this suffices for the purpose of proving \Cref{P:StabOnV}.
\end{rem}

\begin{lem} \label{L:GFc1}
Let $q':\mathcal X'\to \mathcal X$ be a degree $d$ generically finite surjective  morphism of smooth proper integral DM stacks over $\mathbb C$ with projective coarse moduli spaces and let $\alpha$ be a movable curve class on $\mathcal X$.  Let  $\mathcal U\subseteq \mathcal X$ be the maximal open substack over which  $q':\mathcal X'\to \mathcal X$ is finite, and let  $\mathcal U'\subseteq q^{-1}(\mathcal U)$ be any open substack with complement of codimension $\ge 2$ in $q^{-1}(\mathcal U)$.   The class $q'^*\alpha$ is movable (\Cref{L:mov-pullXX}\ref{L:mov-pullXX2}) and we have: 

\begin{enumerate}[label=(\alph*)]
\item \label{L:GFc1-1}    If $\mathcal F'$ and $\mathcal G'$ are torsion-free coherent sheaves on $\mathcal X'$ such that $\mathcal F'|_{\mathcal U'}\cong \mathcal G'|_{\mathcal U'}$, then 
\begin{align}
c_1(\mathcal F')\cdot q'^*\alpha & =c_1(\mathcal G')\cdot q'^*\alpha    \label{E:GFc1FG}\\
\mu_{q'^*\alpha}(\mathcal F')&=\mu_{q'^*\alpha}(\mathcal G').
\end{align}

\item \label{L:GFc1-2} If $\mathcal F$ is a torsion-free coherent sheaf on $\mathcal X$, then for any torsion-free coherent sheaf  $\mathcal F'$ on $\mathcal X'$ such that $\mathcal F'|_{\mathcal U'}\cong q'^*(\mathcal F)|_{\mathcal U'}$, we have 
\begin{align}
c_1(\mathcal F)\cdot \alpha & = \frac{1}{d}(c_1(\mathcal F')\cdot q'^*\alpha ) \label{E:GFc11}\\
\mu_\alpha(\mathcal F)&= \frac{1}{d}\cdot \mu_{q'^*\alpha}(\mathcal F').
\end{align}

\end{enumerate}

\end{lem}

\begin{proof} \ref{L:GFc1-1} 
Let $\mathcal U''\subseteq \mathcal U'$ be the open substack of $\mathcal U'$  over which $\mathcal F$ and $\mathcal G$ are locally free.  Let  $\mathcal Z:=\mathcal X-\mathcal U$, and let  $\mathcal E'_i$ denote the irreducible components of the divisorial locus of $q'^{-1}(\mathcal Z)$.  Note that $q'_*c_1(\mathcal O_{\mathcal X'}(\mathcal E'_i))=0$ (to be clear, here we mean the push-forward in Chow).

Since $\bigwedge^{\operatorname{rk}\mathcal F'} (\mathcal F'|_{\mathcal U''})=\bigwedge^{\operatorname{rk}\mathcal G'} (\mathcal G'|_{\mathcal U''})$, we have that  
$$
(\det (\mathcal F'))\cong (\det (\mathcal G'))  \otimes \mathcal O_{\mathcal X'}(\sum a_i\mathcal E_i') 
$$
for some integers $a_i$.  Then, by virtue of the fact that $q_*c_1(\mathcal O_{\mathcal X'}(\mathcal E'_i))=0$ we can conclude (\Cref{R:EdotMov0}) that
$$
c_1(\mathcal F')\cdot q'^*\alpha = (c_1(\mathcal G')+c_1(\mathcal O_{\mathcal X'}(\sum a_i\mathcal E'_i))\cdot q'^*\alpha
=  c_1(\mathcal G')\cdot q'^*\alpha +c_1(\mathcal O_{\mathcal X'}(\sum a_i\mathcal E'_i))\cdot q'^*\alpha 
$$
$$
=  c_1(\mathcal G')\cdot q'^*\alpha +q'_*c_1(\mathcal O_{\mathcal X'}(\sum a_i\mathcal E'_i))\cdot \alpha  = c_1(\mathcal G')\cdot q'^*\alpha.
$$
The assertion in the lemma on slopes follows since $\operatorname{rk}\mathcal F'=\operatorname{rk}\mathcal G'$. 

\ref{L:GFc1-2}  Clearly we can apply \ref{L:GFc1-1} with $\mathcal G'=(q'^*\mathcal F)^{\vee\vee}$. We obtain  
$$
c_1(\mathcal F')\cdot q'^*\alpha = c_1((q'^*\mathcal F)^{\vee \vee})\cdot q'^*\alpha. 
$$
Applying \ref{L:GFc1-1} now to $q'^*(\det \mathcal F)$ and $\det ((q'^*\mathcal F)^{\vee \vee})$, we obtain 
$$
q'^*c_1(\mathcal F)\cdot q'^*\alpha = c_1((q'^*\mathcal F)^{\vee \vee})\cdot q'^*\alpha.
$$
Putting these together, we have 
$$
c_1(\mathcal F')\cdot q'^*\alpha = q'^*c_1(\mathcal F)\cdot q'^*\alpha = q'_*q'^*c_1(\mathcal F)\cdot \alpha = d \cdot c_1(\mathcal F)\cdot \alpha.
$$
where $d$ is the degree of $q'$.  

For lack of a concrete reference for the last equality on stacks, we include the proof here.  First, it suffices to show that for an effective Cartier divisor $\mathcal D$ on $\mathcal X$, one has $q'_*q'^*\mathcal D=d \mathcal D$.  
Moreover, working with $\mathbb Q$-divisors, then under the isomorphism $\pi_*:\operatorname{CH}^1(\mathcal X)_{\mathbb Q}\to \operatorname{CH}^1(X)_{\mathbb Q}$ to the Chow group of the coarse moduli space, it suffices to work with Cartier divisors on $X$. If we denote also by $q':X'\to X$ the morphism on coarse moduli spaces, then it suffices to show for a Cartier divisor $D$ on $X$, one has $q'_*q'^*D=dD$. 
One can check this equality on the open subset $U=\pi(\mathcal U)\subseteq X$, as it has complement of codimension $\ge 2$ in $X$. In fact, since $X$ is normal, it has singular locus of codimension $\ge 2$, and so we can replace $U$ with the smooth locus $U^\circ\subseteq U$. Restricting $U^\circ$ further if necessary, then using that $X'$ is also normal, we may assume that $q'^{-1}(U^\circ)$ is smooth. Then the morphism $q'|_{q'^{-1}(U^\circ)}:q'^{-1}(U^\circ)\to U^\circ$ is finite and flat. And so the formula follows from the standard statements on cycles on schemes  (e.g., \cite[\href{https://stacks.math.columbia.edu/tag/02RH}{Lem.~02RH}]{stacks-project}).

The assertion on slopes follows since $\operatorname{rk}\mathcal F=\operatorname{rk}\mathcal F'$. 
\end{proof}

Although we will not use it, we record  the following useful consequence for pseudo-effectivity, generalizing  \cite[Lem.~3.5]{CPFol19}: 
\begin{cor}

\label{C:CPL3.5}
Let $q':\mathcal X'\to \mathcal X$ be a degree $d$ generically finite surjective  morphism of smooth proper integral DM stacks over $\mathbb C$ with projective coarse moduli spaces and let $\alpha$ be a movable curve class on $\mathcal X$.  
 Let $\mathcal L$ be a line bundle on $\mathcal X$, and suppose that $\mathcal L'\cong q'^*\mathcal L \otimes \mathcal O_{\mathcal X'}(\mathcal E')$ for some  divisor $\mathcal E'$ supported in $q'^{-1}(\mathcal X-\mathcal U)$, where $\mathcal U$ is the locus over which $q'^{-1}$ is finite.   
 If $\mathcal L'$ is pseudo-effective, then so is $\mathcal L$. 
\end{cor}

\begin{proof}
This follows immediately from \eqref{E:GFc11} of \Cref{L:GFc1}, taking $\mathcal F=\mathcal L$ and $\mathcal F'=\mathcal L'$.  
\end{proof}

We now prove some lemmas that will allow us to prove  \Cref{P:TensorProdStab}.

\begin{lem}[Generically finite pull-back] \label{L:GFstab-1}
Let $q':\mathcal X'\to \mathcal X$ be a generically finite surjective  morphism of smooth proper integral DM stacks over $\mathbb C$ with projective coarse moduli spaces and let $\alpha$ be a movable curve class on $\mathcal X$.    
Then if a torsion-free coherent sheaf  $\mathcal E$ on $\mathcal X$ is not $\alpha$-semi-stable, then $(q'^*\mathcal E)^{\vee \vee}$ is not $q'^*\alpha$-semi-stable on $\mathcal X'$. 
\end{lem}

\begin{proof}
Suppose that $\mathcal F\subseteq \mathcal E$ is a coherent subsheaf of $\mathcal E$ such that $\mu_\alpha(\mathcal F)>\mu_\alpha(\mathcal E)$.  Let $\mathcal U\subseteq \mathcal X$ be the locus over which both $\mathcal F$ and $\mathcal E$ are locally free.  We have a morphism  $q'^*\mathcal F\to q'^*\mathcal E$, which induces a morphism $(q'^*\mathcal F)^{\vee \vee}\to (q'^*\mathcal E)^{\vee\vee}$; since this is generically injective and the sheaves are torsion-free, it is an inclusion.  
 Let $\mathcal F'\subseteq (q'^*\mathcal E)^{\vee \vee}$ be the image 
 of $(q'^*\mathcal F)^{\vee \vee}$ under this morphism.  
We claim that $\mu_{q'^*\alpha}(\mathcal F')>\mu_{q'^*\alpha}((q'^*\mathcal E)^{\vee \vee})$.  
Indeed, since $\mathcal F'|_{q'^{-1}(\mathcal U)}\cong q'^*(\mathcal F|_{\mathcal U})$, we have from \Cref{L:GFc1} that  
$\mu_\alpha(\mathcal F)= \frac{1}{d}\cdot \mu_{q'^*\alpha}(\mathcal F')$, where $d$ is the degree of $q'$. 
The same argument shows that $
\mu_\alpha(\mathcal E)= \frac{1}{d}\cdot \mu_{q'^*\alpha}((q'^*\mathcal E)^{\vee \vee})$.
\end{proof}

For Galois covers, we have a converse, whose proof  is essentially equivalent to the argument in \cite[Prop.~2.6]{CPFol19}: 

\begin{lem}[Galois covers]\label{L:Galstab}
Let $q':\mathcal V'\to \mathcal X$ be a finite flat morphism of smooth proper integral DM stacks over $\mathbb C$ with projective coarse moduli spaces and let $\alpha$ be a movable curve class on $\mathcal X$.   
Suppose there is a finite group $G$ acting on $\mathcal V'$ such that the morphism $q':\mathcal V'\to \mathcal X$ induces an isomorphism   $\mathcal X\cong [\mathcal V'/G]$.  Then  a torsion-free coherent sheaf  $\mathcal E$ on $\mathcal X$ is $\alpha$-semi-stable if and only if $q'^*\mathcal E$ is $q'^*\alpha$-semi-stable. 
\end{lem}

\begin{proof}
If $\mathcal E$ is not $\alpha$-semi-stable, then 
 \Cref{L:GFstab-1} implies that $q'^*\mathcal E$ is not $q'^*\alpha$-semi-stable. 
 For the other direction, assume that $q'^*\mathcal E$ is not $q'^*\alpha$-semi-stable.   Let $\mathcal F'\subseteq q'^*\mathcal E$ be the maximally destabilizing subsheaf (\Cref{P:GKP16-2.24}).  As $\mathcal F'$ is unique, and $q'^*\mathcal E$ is $G$-equivariant, it follows that $\mathcal F'$ is also $G$-equivariant, and so descends to $\mathcal X$.

 To see this, first observe that $q'^*\alpha$ is $G$-invariant.  Therefore, for each $g\in G$, and any torsion-free coherent sheaf $\mathcal G$ on $\mathcal V'$,
 we have $$c_1(g^*\mathcal G)\cdot q'^*\alpha= c_1(g^*\mathcal G)\cdot g^*(q'^*\alpha)=g^*c_1(\mathcal G)\cdot g^*(q'^*\alpha)= c_1(\mathcal G)\cdot q'^*\alpha.$$ 
From this, one can easily check that for each $g\in G$ one has that  $g^*\mathcal F'$ is the maximally destabilizing subsheaf of $g^*q'^*\mathcal E$.
Now, since $q'^*\mathcal E$ is equivariant, we have that for each $g\in G$ we have an isomorphism $\theta_g:g^*q'^*\mathcal E\to q'^* \mathcal E$.  The $q'^*\alpha$-slope of a sheaf is the same for isomorphic sheaves; therefore since maximally destabilizing subsheaves are unique, it must be that $\theta_g$ takes the maximally destabilizing subsheaf, $g^*\mathcal F'$,  of $g^*q'^*\mathcal E$ isomorphically to the maximally destabilizing subsheaf, $\mathcal F'$,  of $q'^*\mathcal E$.  From this one can check that  $\mathcal F'$ is equivariant, and consequently descends to $\mathcal X$; i.e.,   $\mathcal F'\cong q'^*\mathcal F$ for some subsheaf $\mathcal F\subseteq \mathcal E$ on $\mathcal X$. 

Having established that $\mathcal F'$ descends to $\mathcal F$, then by virtue of  \Cref{L:GFc1}\ref{L:GFc1-2}, we see that $$\mu_\alpha(\mathcal F)=\frac{1}{|G|}\mu_{q'^*\alpha}(q'^*\mathcal F)=\frac{1}{|G|}\mu_{q'^*\alpha}(\mathcal F')>\frac{1}{|G|}\mu_{q'^*\alpha}(q'^*\mathcal E)=\mu_\alpha(\mathcal E).$$
 Therefore, $\mathcal E$ is not $\alpha$-semi-stable. 
\end{proof}

One approach to completing the proof of \Cref{P:StabOnV} at this point could be to make an argument similar to \Cref{L:Galstab}, but via faithfully flat descent.  
However, since we will want the following lemma for stacky blow-ups elsewhere,  we prove that first, and then use it to complete the proof of  \Cref{P:StabOnV}.

\begin{lem}[Stacky blow-ups]\label{L:StkyBUstab}
Let $\sigma:\mathcal X'\to \mathcal X$ be a stacky blow-up of a smooth proper integral DM stack over $\mathbb C$ with projective coarse moduli space  and let $\alpha$ be a movable curve class on $\mathcal X$.   A torsion-free coherent sheaf   $\mathcal E$ on $\mathcal X$ is $\alpha$-semi-stable if and only if 
$(\sigma^*\mathcal E)^{\vee \vee}$ is $\sigma^*\alpha$-semi-stable.  
\end{lem}

\begin{proof}   
If $\mathcal E$ is not $\alpha$-semi-stable, then 
 \Cref{L:GFstab-1} implies that $(\sigma^*\mathcal E)^{\vee \vee}$ is not $\sigma^*\alpha$-semi-stable.  Conversely, assume that 
  $(\sigma^*\mathcal E)^{\vee \vee}$ is not $\sigma ^*\alpha$-semi-stable, and let $\mathcal F'\subseteq(\sigma^*\mathcal E)^{\vee \vee}$ be the maximally destabilizing subsheaf.  
We now break the proof into two cases: the case of a blow-up, and the case of a root stack.  In either case, let $\mathcal Z'\subseteq \mathcal X'$ be the locus where $\mathcal F'$ and $(\sigma^*\mathcal E)^{\vee \vee}$ fail to be locally free, and let $\mathcal Z$ be the union of $\sigma (\mathcal Z')$ together with the locus over which $\sigma$ is not finite (in the case where $\sigma$ is a blow-up). Let $\mathcal U:=\mathcal X-\mathcal Z$.    

We claim that there is a coherent subsheaf $\mathcal F\subseteq \mathcal E$ such that $\mathcal F'|_{\sigma^{-1}(\mathcal U)}\cong \sigma^*(\mathcal F|_{\mathcal U})$.  
Note also that $(\sigma^*\mathcal E)^{\vee \vee}|_{\sigma^{-1}(\mathcal U)}\cong \sigma^*(\mathcal E|_{\mathcal U})$. 
Assuming this for the moment, and letting $d$ be the degree of $\sigma$,  then \Cref{L:GFc1}\ref{L:GFc1-2} implies that $$\mu_\alpha(\mathcal F)=\frac{1}{d}\mu_{\sigma^*\alpha}(\mathcal F')>\frac{1}{d}\mu_{\sigma^*\alpha}((\sigma^*\mathcal E)^{\vee \vee})=\mu_\alpha(\mathcal E),$$ completing the proof.

Therefore, let us show the existence of $\mathcal F$. 
For brevity of notation, let $\mathcal U'=\sigma^{-1}(\mathcal U)$ and let   $\sigma'=\sigma|_{\mathcal U'}:\mathcal U'\to \mathcal U$.  We have the inclusion $\mathcal F'|_{\mathcal U'}\hookrightarrow \sigma'^*(\mathcal E|_{\mathcal U})$.  As the push-forward is left exact, and in both cases (blow-ups and root stacks) we have  $\sigma_*\mathcal O_{\mathcal X'}\cong \mathcal O_{\mathcal X}$, then applying $\sigma'_*$, we obtain an inclusion 
$\sigma'_*(\mathcal F'|_{\mathcal U'})\hookrightarrow \sigma'_*(\sigma'^*(\mathcal E|_{\mathcal U})\otimes \mathcal O_{\mathcal U'})= \mathcal E|_{\mathcal U}$.   Let $i:\mathcal U\to \mathcal X$ be the inclusion. Applying $i_*$, we then have $i_*\sigma'_*(\mathcal F'|_{\mathcal U'})\hookrightarrow i_*\mathcal E|_{\mathcal U}=\mathcal E^{\vee \vee}$.  Considering the inclusion $\mathcal E\hookrightarrow \mathcal E^{\vee \vee}$, we take $\mathcal F:=i_*\sigma'_*(\mathcal F'|_{\mathcal U'})\cap \mathcal E$.
Note that by construction, over $\mathcal U$ we have $\mathcal F|_{\mathcal U}=(\sigma_*\mathcal F')|_{\mathcal U}$.

  We now check that  $\mathcal F\subseteq \mathcal E$ satisfies $\mathcal F'|_{\sigma^{-1}(\mathcal U)}\cong \sigma'^*(\mathcal F|_{\mathcal U})$.  
  From the definitions, it suffices to show that the natural morphism $\sigma'^*\sigma'_*(\mathcal F'|_{\mathcal U'})\to \mathcal F'|_{\mathcal U'}$ is an isomorphism.
  First, consider the case of a blow-up.  Then the natural morphism $\sigma^*\sigma_*\mathcal F'\to \mathcal F'$ is an isomorphism on $\sigma^{-1}(\mathcal U)\cong \mathcal U$.  
  Next suppose that $\sigma$ is a root stack.  The same argument works here, but for clarity, we break it down in a slightly different way.  Both $\mathcal F'|_{\sigma^{-1}(\mathcal U)}$ and $(\sigma^*\mathcal E)|_{\sigma^{-1}(\mathcal U)}$ are vector bundles.    
   Since $\mathcal F'|_{\mathcal U'}$ is contained in $\sigma'^*(\mathcal E|_{\mathcal U})$, the action of $\mu_r$ on the fibers of $\mathcal F'|_{\sigma^{-1}(\mathcal U)}$ is trivial (as the action is trivial on the fibers of $\sigma'^*(\mathcal E|_{\mathcal U})$).    
  Thus 
   $\mathcal F'|_{\sigma^{-1}(\mathcal U)}$ descends to a vector bundle $\mathcal F_{\mathcal U}\subseteq \mathcal E|_{\mathcal U}$ on $\mathcal U$; more precisely we have that  the natural morphism $\sigma'^*\sigma'_*(\mathcal F'|_{\mathcal U'})\to \mathcal F'|_{\mathcal U'}$ is an isomorphism (see \Cref{R:VBdescentRS}), 
  which was what we wanted to prove.
\end{proof}

We are now ready to prove \Cref{P:StabOnV}.

\begin{proof}[Proof of \Cref{P:StabOnV}] 
Let $q:V\to \mathcal X$ be a finite flat morphism from a smooth projective variety $V$.  Consider the diagram \eqref{E:TRydh-dgm} and notation from \Cref{T:Rydh}:

\begin{equation}\label{E:P:StabOnVpf}
\xymatrix{
\mathcal V'\ar[r]^{\sigma'} \ar[d]_{q'}^{/G}& V \ar[d]^q\\
[\mathcal V'/G]\cong {\mathcal X'}\ar[r]^<>(0.5)\sigma& \mathcal X
}
\end{equation}
and recall that $\mathcal V'$ is a smooth proper integral DM stack over $\mathbb C$ with projective coarse moduli space, $\sigma$ is a stacky blow-up (a composition of blow-ups and root stacks), $\sigma'$ is generically finite, and $G$ is a finite group acting on $\mathcal V'$.

Now if the vector bundle $\mathcal E$ is not $\alpha$-semi-stable on $\mathcal X$, then $q^*\mathcal E$ is not $q^*\alpha$-semi-stable on $V$, by \Cref{L:GFstab-1}.  Conversely, suppose that $\mathcal E$ is $\alpha$-semi-stable.  
For the sake of contradiction, assume that $q^*\mathcal E$ is not $q^*\alpha$-semi-stable.  
On the one hand,  we have  by \Cref{L:GFstab-1}  that $\sigma'^*q^*\mathcal E$ is not $\sigma'^*q^*\alpha$-semi-stable. 
On the other  hand, 
we can consider pulling back $\mathcal E$ in the other direction around the diagram \eqref{E:P:StabOnVpf}, and, as we are assuming that $\mathcal E$ is $\alpha$-semi-stable, then 
using  \Cref{L:StkyBUstab} and \Cref{L:Galstab}, one has that $q'^*\sigma^*\mathcal E$ is $q'^*\sigma^*\alpha$-semi-stable on $\mathcal V'$.  
 This is a contradiction as $q'^*\sigma^*\mathcal E\cong \sigma'^*q^*\mathcal E$ and  $q'^*\sigma^*\alpha=\sigma'^*q^*\alpha$.
\end{proof}

\subsection{$\alpha$-semi-stability and tensor products of sheaves}

  In this subsection, we generalize  \cite[Thm.~4.2]{GKP16} to DM stacks:

\begin{teo}

\label{T:GKP16-4.2}
 Let $\mathcal X$ be a smooth proper integral DM stack over $\mathbb C$ with projective coarse moduli space, and let $\alpha$ be a movable class of curves on $\mathcal X$.  If $\mathcal F$ and $\mathcal G$ are nonzero 
torsion-free  coherent sheaves on $\mathcal X$, then the following hold:

\begin{enumerate}[label=(\alph*)]

\item $\mu^{\max}_\alpha((\mathcal F\otimes \mathcal G)^{\vee \vee} )=\mu_\alpha^{\max}( \mathcal F)+\mu_\alpha^{\max}( \mathcal G)$. \label{E:GKP16-4.2a}

\item If $\mathcal F$ and $\mathcal G$ are $\alpha$-semi-stable, then $(\mathcal F\otimes \mathcal G)^{\vee \vee}$ is, too. \label{E:GKP16-4.2b}

\end{enumerate}
\end{teo}

Before the proof of the theorem, we establish one more lemma, which generalizes \cite[Lem.~4.6]{GKP16}:

\begin{lem}[Blow-ups and tensor products]

\label{L:BU-tens} 
Let $\sigma:\mathcal X'\to \mathcal X$ be a blow-up of a smooth proper integral  DM stack over $\mathbb C$ with projective coarse moduli space  and let $\alpha$ be a movable curve class on $\mathcal X$.   For nonzero torsion-free coherent sheaves   $\mathcal E_1$ and  $\mathcal E_2$ 
 on $\mathcal X$,
we have $(\mathcal E_1\otimes \mathcal E_2)^{\vee \vee}$ 
 is $\alpha$-semi-stable if and only if 
$$((\sigma^*\mathcal E_1)^{\vee \vee} \otimes (\sigma^*\mathcal E_2)^{\vee \vee})^{\vee \vee}$$ is $\sigma^*\alpha$-semi-stable.  
\end{lem}

\begin{proof}
For brevity,  will use the standard $[-]$ and $\boxtimes$ notation to indicate reflexive hulls.  In other words, 
$(\mathcal E_1\otimes \mathcal E_2)^{\vee \vee} = \mathcal E_1\boxtimes \mathcal E_2$ 
and $((\sigma^*\mathcal E_1)^{\vee \vee} \otimes (\sigma^*\mathcal E_2)^{\vee \vee})^{\vee \vee}=(\sigma^{[*]}\mathcal E_1)\boxtimes (\sigma^{[*]}\mathcal E_2)$.  To start, from \Cref{L:StkyBUstab}, we have that $ \mathcal E_1\boxtimes \mathcal E_2$ is $\alpha$-semi-stable if and only if $\sigma^{[*]}( \mathcal E_1\boxtimes \mathcal E_2)$ is $\sigma^*\alpha$-semi-stable.  So we aim to show that $\mathcal E':=\sigma^{[*]}( \mathcal E_1\boxtimes \mathcal E_2)$ is $\sigma^*\alpha$-semi-stable if and only if $\mathcal E'':=(\sigma^{[*]}\mathcal E_1)\boxtimes (\sigma^{[*]}\mathcal E_2)$ is $\sigma^*\alpha$-semi-stable.  

Let $\mathcal Z$ be the locus in $\mathcal X$ over which $\mathcal E_1$ and $\mathcal E_2$ fail to be locally free or 
$\sigma$ fails to be an isomorphism; this locus is codimension $\ge 2$.  Let $\mathcal U:=\mathcal X-\mathcal Z$, and let $i:\mathcal U\hookrightarrow \mathcal X$ and  $i':\sigma ^{-1}(\mathcal U)\hookrightarrow \mathcal X'$ be the inclusions.  Clearly we have $\mathcal E'|_{\sigma^{-1}(\mathcal U)}\cong \mathcal E''|_{\sigma^{-1}(\mathcal U)}$.   Given any subsheaf $\mathcal F'\subseteq \mathcal E'$, we then obtain a subsheaf $\mathcal F'|_{\sigma^{-1}(\mathcal U)}\subseteq \mathcal E''|_{\sigma^{-1}(\mathcal U)}$. We then obtain an inclusion $ i'_*(\mathcal F'|_{\sigma^{-1}(\mathcal U)})\subseteq i'_*(\mathcal E''|_{\sigma^{-1}(\mathcal U)})$, and finally via the canonical inclusion $\mathcal E''\to i'_*i'^*\mathcal E''$, and setting $\mathcal F''= i'_*(\mathcal F'|_{\sigma^{-1}(\mathcal U)})\cap \mathcal E''$, we obtain a subsheaf $\mathcal F''\subseteq \mathcal E''$, with $\mathcal F''|_{\sigma^{-1}(\mathcal U)}\cong \mathcal F'|_{\sigma^{-1}(\mathcal U)}$.  Interchanging the roles of $\mathcal E'$ and $\mathcal E''$,  we also have that given any 
 subsheaf $\mathcal F''\subseteq \mathcal E''$ we obtain a 
 subsheaf $\mathcal F'\subseteq \mathcal E'$,
  with $\mathcal F'|_{\sigma^{-1}(\mathcal U)}\cong \mathcal F''|_{\sigma^{-1}(\mathcal U)}$.

With this set-up, we see from \Cref{L:GFc1}\ref{L:GFc1-1} 
 that $\mu_{\sigma^*\alpha}(\mathcal E')=\mu_{\sigma^*\alpha}(\mathcal E'')$, and moreover, given any subsheaf $\mathcal F'\subseteq \mathcal E'$ (resp.~$\mathcal F''\subseteq \mathcal E''$), we obtain a  subsheaf $\mathcal F''\subseteq \mathcal E''$ (resp.~$\mathcal F'\subseteq \mathcal E'$) with $\mu_{\sigma^*\alpha}(\mathcal F'')=\mu_{\sigma^*\alpha}(\mathcal F')$ (resp.~$\mu_{\sigma^*\alpha}(\mathcal F')=\mu_{\sigma^*\alpha}(\mathcal F'')$).  Clearly $\mathcal E'$ is $\sigma^*\alpha$-semi-stable if and only if $\mathcal E''$ is $\sigma^*\alpha$-semi-stable.  
\end{proof}

\begin{proof}[Proof of \Cref{T:GKP16-4.2}] 
 We start by proving part \ref{E:GKP16-4.2b}.  
 By \Cref{T:Rossi} applied to $\mathcal F\oplus \mathcal G$, 
   there exists a blow-up $\mu:\mathcal X'\to \mathcal X$ such that $(\mu^*\mathcal F)^{\vee\vee}$ and $(\mu^*\mathcal G)^{\vee \vee}$ are vector bundles.  Moreover, by \Cref{L:StkyBUstab},  $(\mu^*\mathcal F)^{\vee\vee}$ and $(\mu^*\mathcal G)^{\vee \vee}$ are $\mu^*\alpha$-semi-stable. Therefore, by \Cref{P:TensorProdStab},  the tensor product 
  $(\mu^*\mathcal F)^{\vee\vee}\otimes (\mu^*\mathcal G)^{\vee \vee}$ is $\mu^*\alpha$-semi-stable. 
  It follows from \Cref{L:BU-tens} that $(\mathcal F\otimes \mathcal G)^{\vee \vee}$ is $\alpha$-semi-stable.

We now prove \ref{E:GKP16-4.2a}.  
This is the same as the proof that appears in \cite{GKP16}, using what we have proved so far for DM stacks; we include it here for completeness. One shows \ref{E:GKP16-4.2a} of the theorem in three steps. 

\subsubsection*{Step 1: Some reductions:} Since numerical classes and slopes are unaffected when modifying $\m{F}$ and $\m{G}$ along a subset of codimension at least two, we are free to replace these sheaves by their double duals, and assume henceforth that the sheaves $\m{F}$ and $\m{G}$ are reflexive.

Combining the Harder--Narasimhan filtration and a Jordan--H\"older-filtration as in
Remark~\ref{rem:RHNF}, we choose a filtration of $\m{F}$, say $0 = \m{F}_0 \subsetneq
\m{F}_1 \subsetneq \cdots \subsetneq \m{F}_k = \m{F}$ such that the following holds.
\begin{enumerate}[label=(\arabic*)]
\item\label{il:1} The quotients $\mathcal{Q}_{i+1} := \m{F}_{i+1}/\m{F}_i$ are torsion-free
  and $\alpha$-stable for all $i$.
  \item\label{il:4} $\mu_\alpha(\mathcal Q_1)=\mu_\alpha^{\max}(\mathcal F)$.
\item\label{il:3} The sequence of slopes, $\bigl( \mu_{\alpha}(\mathcal{Q}_i) \bigr)_{1 \leq i \leq  k}$, is decreasing.
\item\label{il:2} The sequence of ranks, $(\operatorname{rk} \m{F}_i)_{0 \leq i \leq k}$, is
  strictly increasing.
\end{enumerate}
Taking reflexive tensor products with $\m{G}$, denoted $(-)\boxtimes \mathcal G$, we obtain a filtration of $\m{F}
\boxtimes \m{G}$,
\begin{equation}\label{eq:FTP}
  0 = \m{F}_0 \boxtimes \m{G} \subsetneq \m{F}_1 \boxtimes \m{G} \subsetneq \cdots \subsetneq \m{F}_k \boxtimes \m{G} = \m{F} \boxtimes \m{G}.
\end{equation}
There exists an open substack $\m{U} \subseteq \X$ with complement of codimension $\geq 2$, where all sheaves $\m{F}$, $\m{G}$, $\m{F}_i$ and $\m{F}_i \boxtimes \m{G}$, as well as all quotients $\mathcal{Q}_i$ and $(\m{F}_{i+1}\boxtimes \m{G})/(\m{F}_i\boxtimes \m{G})$ are locally free. Pushing forward along this open embedding, we obtain that each term $\m{F}_i \boxtimes \m{G}$ is saturated  in $\m{F}_{i+1} \boxtimes \m{G}$. Moreover, on this open set, since two
reflexive sheaves together with a given morphism are isomorphic via that morphism if and only if they agree away from codimension $2$, the
quotients given by the filtration~\eqref{eq:FTP} can be identified as follows,
\begin{equation}\label{eq:XP2}
  \left( \cfrac{\m{F}_{i+1}\boxtimes \m{G}}{\m{F}_i\boxtimes \m{G}} \right)^{\vee \vee} = \mathcal{Q}_{i+1} \boxtimes \m{G}.
\end{equation}
The slopes of these sheaves are computed as follows.
  \begin{align}
    \mu_{\alpha}(\mathcal{Q}_{i+1} \boxtimes \m{G}) & = \mu_{\alpha}(\mathcal{Q}_{i+1}) + \mu_{\alpha}(\m{G}) && \text{(slope of product)}\\
    & \leq \mu_{\alpha}(\mathcal{Q}_1) + \mu_{\alpha}(\m{G}) && \text{(by~\ref{il:3}, above)} \\
    & \leq \mu_{\alpha}^{\max}(\m{F}) + \mu_{\alpha}^{\max}(\m{G}) && 
\text{(by~\ref{il:4}, above)} 
    \label{eq:fx3}
  \end{align}

\subsubsection*{Step 2: The case where $\m{F}$ or $\m{G}$ is $\alpha$-semi-stable}

The roles of $\m{F}$ and $\m{G}$ being symmetric, consider the case where $\m{G}$ is $\alpha$-semi-stable. 
Then, using part \ref{E:GKP16-4.2b} of  \Cref{T:GKP16-4.2}, proven above, the quotient sheaves $\mathcal{Q}_i \boxtimes \m{G}$ will thus
be $\alpha$-semi-stable. 
Consequently,  \Cref{C:GKP16-2.17} and  \eqref{eq:fx3}  imply that for any $\alpha$-semi-stable sheaf $\mathcal A$ with  $\mu_{\alpha}(\m{A}) > \mu_{\alpha}^{\max}(\m{F}) + \mu_{\alpha}^{\max}(\m{G})$, then any morphism
$\m{A} \ra \mathcal{Q}_i \boxtimes \m{G}$ will be zero.  Considering the filtration \eqref{eq:FTP}, it  follows that any
morphism $\m{A} \ra \m{F} \boxtimes \m{G}$ will be zero (start with $i=k$, and proceed by induction), and therefore, that the $\alpha$-slope of any
$\alpha$-semi-stable subsheaf $\m{B} \subseteq \m{F} \boxtimes \m{G}$ is bounded by the inequality: 
$$
\mu_{\alpha}(\m{B}) \leq \mu^{\max}_{\alpha}(\m{F})+\mu^{\max}_{\alpha}(\m{G}).
$$
As the maximally destabilizing subsheaf (\Cref{P:GKP16-2.24}) of $\mathcal F\boxtimes \mathcal G$  is $\alpha$-semi-stable, this implies that 
$$
\mu_\alpha^{\max}(\mathcal F\boxtimes \mathcal G)\le \mu_\alpha^{\max}(\mathcal F)+\mu_\alpha^{\max}(\mathcal G).
$$
On the other hand, $\m{F} \boxtimes \m{G}$  contains the subsheaf
$\m{F}_1 \boxtimes \m{G} = \mathcal{Q}_1 \boxtimes \m{G}$, whose slope equals
$\mu^{\max}_{\alpha}(\m{F}) + \mu_{\alpha}(\m{G}) = \mu^{\max}_{\alpha}(\m{F}) + \mu^{\max}_{\alpha}(\m{G})$.  This
proves \ref{E:GKP16-4.2a} of the theorem, in the case where one of the factors is $\alpha$-stable.

\subsubsection*{Step 3: The general case}

Recalling from \ref{il:1}, above, that the quotient sheaves $\mathcal{Q}_{i+1}$ are $\alpha$-stable, we
can apply  Step~2 to conclude that
$$
\mu^{\max}_{\alpha} \bigl(\mathcal{Q}_{i+1} \boxtimes \m{G} \bigr) = \mu^{\max}_{\alpha} \bigl(\mathcal{Q}_{i+1}
\bigr) + \mu^{\max}_{\alpha} \bigl(\m{G} \bigr) \leq\mu^{\max}_{\alpha} \bigl(\m{F} \bigr) +
\mu^{\max}_{\alpha} \bigl(\m{G} \bigr).
$$
As in Step 2, this implies that the $\alpha$-slope of any $\alpha$-semi-stable subsheaf $\m{B} \subseteq \m{F} \boxtimes
\m{G}$ is bounded by the sum $\mu^{\max}_{\alpha}(\m{F})+\mu^{\max}_{\alpha}(\m{G})$.  
As the maximally destabilizing subsheaf of $\mathcal F\boxtimes \mathcal G$ is $\alpha$-semi-stable, we obtain that 
$$
\mu_\alpha^{\max}(\mathcal F\boxtimes \mathcal G)\le \mu_\alpha^{\max}(\mathcal F)+\mu_\alpha^{\max}(\mathcal G).  
$$
On the other hand,
$\m{F} \boxtimes \m{G}$  contains the reflexive product of the maximally
destabilising subsheaves of $\mathcal F$ and $\mathcal G$.  The slope of this product equals $\mu^{\max}_{\alpha}(\m{F}) +
\mu^{\max}_{\alpha}(\m{G})$.  This proves \ref{E:GKP16-4.2a} of the theorem,  in general.
\end{proof}

From this one has:

\begin{cor}

\label{T:SonCP-T5}
Let $\alpha \in  N_1(\mathcal X)_{\mathbb R}$ be a movable class. If $\mathcal F$ and $\mathcal G$ are torsion-free and $\alpha$-semistable
coherent sheaves on $\mathcal X$, then their tensor product
$$
\mathcal F\hat \otimes \mathcal G := (\mathcal F\otimes \mathcal G)/(\mathcal F\otimes \mathcal G)_{tor}
$$
modulo torsion, is again $\alpha$-semistable, and $\mu_\alpha(\mathcal F \hat \otimes \mathcal G) = \mu_\alpha(\mathcal F)+\mu_\alpha(\mathcal G)$.
\end{cor}

\begin{proof}
Since $\mathcal F\hat \otimes \mathcal G$  and its reflexive
hull are isomorphic outside a closed subvariety of codimension $\ge 2$, the assertion follows from \Cref{T:GKP16-4.2}\ref{E:GKP16-4.2a}.  The formula
for the $\alpha$-slope of $\mathcal F\hat \otimes \mathcal G$
is  valid for arbitrary nonzero torsion-free coherent sheaves $\mathcal F$ and
$\mathcal G$.
\end{proof}

\begin{cor}\label{C:GKP16-4.2}
Let $\alpha \in  N_1(\mathcal X)_{\mathbb R}$ be a movable class. If $\mathcal F$ is a torsion-free coherent sheaf on $\mathcal X$, and $\mathcal F^{\hat \otimes N}$ has a subsheaf with positive $\alpha$-slope, then $\mathcal F$ has a subsheaf with positive $\alpha$-slope.
\end{cor}

\begin{proof}
This follows directly from \Cref{T:GKP16-4.2}\ref{E:GKP16-4.2a}.
\end{proof}

We also record the following variation on \Cref{T:GKP16-4.2} following \cite[Prop.~2.8, Thm.~2.9]{CPFol19}:

\begin{teo}\label{T:GKP16-4.2Sym}
In the notation of  \Cref{T:GKP16-4.2}, we have for all integers $m\ge 1$:

\begin{enumerate}[label=(\alph*)]

\item $\mu^{\max}_\alpha((\operatorname{Sym}^m\mathcal F)^{\vee \vee} )=m\cdot \mu_\alpha^{\max}( \mathcal F)$ and $\mu_\alpha^{\max}((\bigwedge^m\mathcal F)^{\vee\vee})= m\cdot \mu_\alpha^{\max}(\mathcal F)$.  
 \label{E:GKP16-4.2aSym}

\item If $\mathcal F$ is $\alpha$-semi-stable, then $(\operatorname{Sym}^m\mathcal F)^{\vee \vee}$ and $(\bigwedge^m\mathcal F)^{\vee\vee}$ are, too. \label{E:GKP16-4.2bSym}
\end{enumerate}
\end{teo}

\begin{proof}
The proof is identical to that of  \Cref{T:GKP16-4.2}, adjusting the argument for the symmetric and anti-symmetric products, and using the identities:
\begin{align*}
c_1((\operatorname{Sym}^m\mathcal F)^{\vee \vee} )& = m\cdot \frac{\operatorname{rk}(\operatorname{Sym}^m\mathcal F)}{\operatorname{rk}(\mathcal F)}\cdot c_1(\mathcal F)\\
c_1((\bigwedge^m\mathcal F)^{\vee\vee}) & =m\cdot \frac{\operatorname{rk}(\bigwedge^m\mathcal F)}{\operatorname{rk}(\mathcal F)}\cdot c_1(\mathcal F).
\end{align*}
(These are standard; for the first identity above, see e.g., \cite[Lem.~1.5]{Vkod1}. For the second, use the splitting principle, and then check that if $r=\operatorname{rk}\mathcal F$, then $m\cdot \frac{\operatorname{rk}(\bigwedge^m\mathcal F)}{\operatorname{rk}(\mathcal F)}= m\binom{r}{m}/r=\binom{r-1}{m-1}$.) For brevity, we leave the details of the remainder of the  proof to the reader.
\end{proof}

\section{Chern class inequalities}\label{S:ChernClasses}

We now establish some inequalities in Chern classes for smooth DM stacks, generalizing some standard inequalities that hold on smooth projective varieties.  

\subsection{Bogomolov--Gieseker inequalities}\label{S:BG}

First we prove \Cref{T:BGmain}.

\begin{proof}[Proof of \Cref{T:BGmain}]
Let $q:V\to \mathcal X$ be a finite flat morphism from a smooth projective surface $V$.  From \Cref{P:StabOnV}, we have that $q^*\mathcal E$ is $q^*\alpha$-semi-stable on $V$.  Therefore, using  \cite[Thm.~5.1]{GKP16}, one has that  $2r\cdot c_2(q^*\mathcal E)-(r-1)\cdot c_1(q^*\mathcal E)^2\ge 0$.  Consequently, we have (using a more precise formulation for exposition): 
\begin{align*}
0&\le \deg (2r\cdot c_2(q^*\mathcal E)\cap [V]-(r-1)\cdot c_1(q^*\mathcal E)^2\cap [V])\\
&= \deg (2r\cdot q^*c_2(\mathcal E)\cap [V]-(r-1)\cdot q^*c_1(\mathcal E)^2\cap [V])\\
&= \deg(2r\cdot c_2(\mathcal E)\cap q_*[V]-(r-1)\cdot c_1(\mathcal E)^2\cap q_*[V])\\
& = \deg(\deg (q) (2r\cdot c_2(\mathcal E)\cap [\mathcal X]-(r-1)\cdot c_1(\mathcal E)^2\cap [\mathcal X]))
\end{align*}
where $\deg:\operatorname{CH}_0(-)\to \mathbb Z$ is push-forward to a point. 
\end{proof}

Next we prove a generalization of \Cref{T:BGmain} to higher dimensions by restricting the types of moving classes we allow. This generalizes the standard result, e.g.,   \cite[Thm.~7.3.1]{huybrechts_lehn_2010}, in the case of smooth projective varieties, and provides an alternate proof of  the  special case of   \cite[Thm.~A.2]{JK24BGS} over $\mathbb C$: 

\begin{teo}[Jiang--Kundu]\label{T:BGn}
Let $\mathcal X$ be a smooth proper integral Deligne--Mumford stack of dimension $n$ over $\mathbb C$  with projective coarse moduli space $\pi:\mathcal X\to X$, and let $H$ be an ample line bundle on $X$.  The class $\alpha :=c_1(\pi^*H)^{n-1}$ is a movable   class in $\operatorname{N}_1(\mathcal X)_{\mathbb R}$, and if  $\mathcal E$ is an $\alpha$-semi-stable vector bundle of rank $r$ on $\mathcal X$,  then
$$
(2r\cdot c_2(\mathcal E)-(r-1)\cdot c_1(\mathcal E)^2)\cdot c_1(\pi^*H)^{n-2}\ge 0.
$$

\end{teo}

\begin{proof} Let $q:V\to \mathcal X$ be a finite flat morphisms from a smooth projective variety $V$.  Note that for any $i$ we  have $q^*(c_1(\pi^*H)^{i})= (q^*c_1(\pi^*H))^{i} = (c_1(q^*\pi^*H))^{i}$. 
First we confirm that $\alpha$ is a moving class.  For this, it suffices from \Cref{L:mov-pullXX}\ref{L:mov-pullXX2} to show that $q^*\alpha$ is movable (where $q^*$ is the numerical pull back).  From \Cref{R:Fpb=Npb}, we have that the numerical pull back agrees with the flat pull back. In particular, from the short computation in Chern classes we just did, we have that $q^*\alpha = (c_1(q^*\pi^*H))^{n-1}$.  Since $\pi\circ q$ is finite, we have that $q^*\pi^*H$ is ample on $V$, and so $q^*\alpha$ is a movable class. 

 Next, to prove the inequality in Chern classes, we observe that from \Cref{P:StabOnV}, we have that $q^*\mathcal E$ is $q^*\alpha$-semi-stable on $V$. Therefore, from \cite[Thm.~7.3.1]{huybrechts_lehn_2010}, we have that 
 $$
(2r\cdot c_2(q^*\mathcal E)-(r-1)\cdot c_1(q^*\mathcal E)^2)\cdot c_1(q^*\pi^*H)^{n-2}\ge 0.
$$
An argument similar to that in the proof of \Cref{T:BGmain} gives the desired inequality on $\mathcal X$. 
\end{proof}

\subsection{Miyaoka--Yau type inequalities}\label{S:MY}

 Denoting $c_i(\mathcal X):=c_i(\mathcal T_{\mathcal X})$, then 
if $\mathcal T_{\mathcal X}$ is $\alpha$-semi-stable,  we have from \Cref{T:BGmain} that 
\begin{equation}\label{E:WeakMY}
4c_2(\mathcal X)-c_1(\mathcal X)^2\ge 0.
\end{equation}
  Recall that for a smooth projective surface $V$ with $K_V$ big and nef, one has that $T_V$ is $\alpha$-semi-stable for $\alpha=c_1(K_V)$ (see the references around \cite[Thm.~7.1]{GKPT19} for this and related results).  However, in general, it is not clear to us when $\mathcal T_{\mathcal X}$ would be $\alpha$-semi-stable.  Moreover, the inequality \eqref{E:WeakMY} may not be optimal. 
  
   For instance, if $\mathcal X=[V/G]$ for a smooth projective surface $V$ with $K_V$ big,   and $q:V\to \mathcal X=[V/G]$ is the quotient map, then since $q^*\mathcal T_{\mathcal X}=T_V$, one can easily, via pull-back, use the Miyaoka--Yau inequality $3c_2(V)-c_1(V)^2\ge 0$ \cite[Thm.~4]{miyaoka77} to obtain the stronger inequality that $$3c_2(\mathcal X)-c_1(\mathcal X)^2\ge 0.$$

In fact, in any dimension,  if $\mathcal X=[V/G]$ for a smooth projective variety $V$ of dimension $n$ with $K_V$ big and nef   and $q:V\to \mathcal X=[V/G]$ is the quotient map, then again, since $q^*\mathcal T_{\mathcal X}=T_V$, one can 
easily, via pull-back, use the Miyaoka--Yau inequality in higher dimensions (see e.g., \cite[Thm.~1.1]{GKPT19}, for references),     $(2(n-1)\cdot c_2(V)-n\cdot c_1(V)^2)\cdot c_1(K_V)^{n-2}\ge 0$, 
  to obtain  
\begin{equation}\label{E:MY}
(2(n-1)\cdot c_2(\mathcal X)-n\cdot c_1(\mathcal X)^2)\cdot c_1(K_{\mathcal X})^{n-2}\ge 0.
\end{equation}
Regarding coarse moduli spaces, we note that in general, if $\mathcal X$ is a smooth proper DM stack over $\mathbb C$ with projective coarse moduli space $X$, and $K_X$ is big and nef, then \eqref{E:MY} holds with $\mathcal X$ replaced with $X$ (this follows from \cite[Thm.~1.1]{GKPT19}; as $X$ has finite quotient singularities, this also follows from some earlier work referenced there, as well).

%%%%%%%%%%%%%%%%%%%%%%%%%%%%%%%%%%%%%%%%%%%%%%%%%%%%%%%%%%%%%%%%%%%%%%%%%%%%%%%%%%%%%% APPENDIX -- ARXIV VERSION ONLY %%%%%%%%%%%%%%%%%%%%%%%%%%%%%%%%%%%%%%%%%%%%%%%%%%%%%%%%%%%%%

%%%%%%%%%%%%%%%%%%%%%%%%%%%%%%%%%%%%%%%%%%%%%%%%%%%%%%%%%%%%%%%%%%%%%%%%%%%%%%%%%%%%%% START AXRIV IF %%%%%%%%%%%%%%%%%%%%%%%%%%%%%%%%%%%%%%%%%%%%%%%%%%%%%%%%%%%%%
\ifArxiv

\appendix

\section{Resolution of singularities}\label[appendix]{S:ResSing}

In this appendix we provide a proof of \Cref{T:Hironaka}.

\begin{proof}[Proof of \Cref{T:Hironaka}]
 The standard proof  we give to establish a proper morphism $\mu:\mathcal X'\to \mathcal X$ satisfying  \ref{E:Hir-1} and \ref{E:Hir-2} holds verbatim for Artin stacks.   Consider a smooth  groupoid $s,t:R \rightrightarrows U$ with $ \mathcal X=[R\rightrightarrows U]$ as the stack quotient.  Let $D\subseteq U$ be the pull back of the divisor $\mathcal D$.   The functorial resolution of singularities (e.g., \cite{BM08Res, EV98Res, temkin12, temkin18}) provide a log resolution of singularities of $(U,D)$ (resp.~$(R,s^*D=t^*D)$)  by a finite sequence of blow-ups $U_{i+1}\to U_i$ (resp.~$R_{i+1}\to R_i$) with smooth centers, with $U_0=U$ (resp.~$R_0=R$).  Moreover, the functoriality of the sequence of blow-ups with respect to smooth morphisms gives more.
 To begin, the center of the blow-up for $R_1\to R_0$ is equal to the pull-back of the center for the blow-up for $U_1\to U_0$ along $s$, as well as along $t$ (as this is true for the pull-back along any smooth morphism $R_0\to U_0$).  Now, since blow-ups are also functorial for flat morphisms (e.g., \cite[\href{https://stacks.math.columbia.edu/tag/0805}{Lem.~0805}]{stacks-project}), we have a cartesian diagram
 $$
\xymatrix{
R_{1} \ar[r] \ar@<-.5ex>@{->}[d]_{t_{1}}\ar@<.5ex>@{->}[d]^{s_{1}}& R_0 \ar@<-.5ex>@{->}[d]_{t_{0}}\ar@<.5ex>@{->}[d]^{s_{0}}\\
U_{1}\ar[r]& U_0
}
$$
where $t_1$ and $s_1$ are defined by base change, and are therefore smooth.   The morphism $U_1\to U_0$ being a blow-up, is projective, and similarly for $R_1\to R_0$.  
Continuing on inductively, we have cartesian diagrams of blow-ups of smooth groupoids 
 $$
\xymatrix{
 \cdots \ar[r] &R_{i+1} \ar[r] \ar@<-.5ex>@{->}[d]_{t_{i+1}}\ar@<.5ex>@{->}[d]^{s_{i+1}}& R_i \ar@<-.5ex>@{->}[d]_{t_{i}}\ar@<.5ex>@{->}[d]^{s_{i}} \ar[r] & \cdots \\
\cdots \ar[r]& U_{i+1}\ar[r]& U_i \ar[r]& \cdots.
}
$$
If this process stops after $n$ steps, then the morphism $\mathcal X':=[R_n \rightrightarrows U_n]\to \mathcal X=[R \rightrightarrows U]$ satisfies conditions \ref{E:Hir-1} and \ref{E:Hir-2} of the theorem. 

Next we show that the morphism $\mu$ is schematic and projective.  The first step is to show that $\mu$ is representable (by algebraic spaces).  For this we consider the fibered product diagram
$$
\xymatrix{
U_n \ar[r] \ar[d]& U_0 \ar[d]\\
\mathcal X'\ar[r]^\mu& \mathcal X
}
$$
Since $U_0\to \mathcal X$ is an \'etale cover, and $U_n$ is a scheme, we can conclude that $\mu$ is representable \cite[\href{https://stacks.math.columbia.edu/tag/04XB}{\S 04XB}]{stacks-project}.  We now show that $\mu$ is schematic and projective. For this, consider a scheme $S$, a morphism $S\to \mathcal X$,  and the fibered product diagram 
$$
\xymatrix{
S' \ar[r] \ar[d]& S\ar[d]\\
\mathcal X'\ar[r]^\mu& \mathcal X
}
$$
The claim is that the algebraic space $S'$ is a scheme, and that the morphism $S'\to S$ is projective.  It suffices to prove the latter claim, which can be checked inductively on each blow-up, using the blow-up construction on the stacks to construct a relatively ample line bundle for $S'/S$, as follows.  Denote by $\mu_{i+1}:\mathcal X_{i+1}\to \mathcal X_i$ the blow-up at the $(i+1)$-th step, and let $S_{i+1}\to S_i$ be the morphism obtained by base change, with $S=S_0$.  Let $\mathcal E_{i+1}$ be the exceptional divisor of the blow-up $\mu_{i+1}$, and let $E_{i+1}'$ be the pull-back of $\mathcal E_{i+1}$ to $S_{i+1}$.  The claim is that $-E_{i+1}$ is relatively ample over $S_{i}$, so that we are done inductively.  From the standard lemma below, it suffices to check that $-E_{i+1}$ is ample on every fiber (e.g., \Cref{P:amp-fam}).  

  Since we are working with fibers over points $x:\operatorname{Spec}\mathbb C\to S_i$, these points lift to $U_i$, to give a commutative diagram
 $$
 \xymatrix{
 U_i \ar[d]& \operatorname{Spec}\mathbb C \ar[l]_x \ar[d]^x\\
 \mathcal X& S_i \ar[l].
 }
 $$
 Therefore, we have a commutative diagram
$$
\xymatrix@R=1em{
U_{i+1,x} \ar[d]\ar[r]& \operatorname{Spec}\mathbb C \ar[d]^x\\
U_{i+1} \ar[d]\ar[r]& U_{i}\ar[d]\\
\mathcal X_{i+1}\ar[r]& \mathcal X_i
}
$$
where the fiber $U_{i+1,x}\to \operatorname{Spec}\mathbb C$ of $U_{i+1}\to U_i$ is equal to the fiber $S_{i+1,x}\to \operatorname{Spec}\mathbb C$ of $S_{i+1}\to S_i$.  Moreover, the restriction $E'_{i+1}|_{S_{i+1,x}}$ to the fiber is isomorphic to the pull-back of $\mathcal E_{i+1}$ to $U_{i+1}$ restricted to the fiber $U_{i+1,x}$.  
Since the morphism $U_{i+1}\to U_i$ is a blow-up of schemes, and so projective, the fibers are projective, and moreover, the pull-back of $\mathcal O_{\mathcal X_{i+1}}(\mathcal E_{i+1})^{-1}$ to $U_{i+1}$ is a positive twist of the relative $\mathcal O(1)$ of the blow-up, and so is ample on the fibers. 

The last thing left to check is that  if $\mathcal X$ has a quasi-projective coarse moduli space, then so does $\mathcal X'$.   
This immediately follows from \cite[Lem.~2.2]{KTbir23} and the fact that $\mu:\mathcal X'\to \mathcal X$ is projective;   we include a different proof here for convenience. 
Since projectivity does not satisfy descent, we cannot simply conclude using the fact that the morphisms $U_{i+1}\to U_i$ are projective.
Nevertheless, the quasi-projectivity of the coarse moduli space can be checked inductively on each blow-up, using the blow-up construction on the stacks to construct a relatively ample line bundle on the coarse moduli spaces.  The argument is essentially identical to what is above, but we include it for completeness.  Let $\pi_i:\mathcal X_i\to X_i$ be the  coarse moduli spaces, where we are assuming that $X_0$ is quasi-projective.  Denote by $\mu_{i+1}:\mathcal X_{i+1}\to \mathcal X_i$ the blow-up at the $(i+1)$-th step, and by $m_{i+1}:X_{i+1}\to X_i$ the induced map of coarse moduli spaces.  Let $\mathcal E_{i+1}$ be the exceptional divisor of the blow-up $\mu_{i+1}$, and assume that $\mathcal O_{\mathcal X_{i+1}}(\mathcal E_{i+1})^{\otimes n_{i+1}}$ descends to a line bundle $E_{i+1}$ on $X_{i+1}$.  The claim is that $E_{i+1}^{-1}$ is relatively ample over $X_{i}$, so that inductively, using that $X_i$ is quasi-projective, we get that $X_{i+1}$ is also quasi-projective.   From the standard lemma below, it suffices to check that $E_{i+1}^{-1}$ is ample on every fiber (e.g., \Cref{P:amp-fam}).  

  Since we are working with fibers over points $\operatorname{Spec}\mathbb C\to X_i$, these points lift, and we have a diagram
$$
\xymatrix@R=1em{
U_{i+1,x} \ar[d]\ar[r]& \operatorname{Spec}\mathbb C \ar[d]^x\\
U_{i+1} \ar[d]\ar[r]& U_{i}\ar[d]\\
\mathcal X_{i+1}\ar[d]\ar[r]& \mathcal X_i\ar[d]\\
X_{i+1}\ar[r]& X_i\\
}
$$
Since the morphism $U_{i+1}\to U_i$ is a blow-up of schemes, and so projective, the fibers are projective, and moreover, the pull-back of $\mathcal O_{\mathcal X_{i+1}}(\mathcal E_{i+1})^{\otimes -n_{i+1}}$ to $U_{i+1}$ is a positive twist of the relative $\mathcal O(1)$ of the blow-up, and so is ample on the fibers.  Consequently, we have a quasi-finite morphism 

$$
U_{i+1,x}\to X_{i+1,x}
$$
from a projective scheme to a proper algebraic space, with the property that the pull-back of the line bundle to $E_{i+1}^{-1}|_{X_{i+1,x}}$ on the fiber $X_{i+1,x}$ to the fiber $U_{i+1,x}$ is ample.  From \Cref{L:amp-fin} (use that quasi-finite plus proper implies finite for algebraic spaces via descent for finite morphisms), we have that $E_{i+1}^{-1}|_{X_{i+1,x}}$ is ample, completing the proof. 
\end{proof}

\subsection{Ample line bundles on algebraic spaces}
For lack of a suitable reference, we include a few results here regarding ample line bundles on algebraic spaces that will will want to use. To start,   
we will say that a line bundle $L$ on an algebraic space $X$ is ample if $X$ is a scheme and $L$ is ample in the usual sense \cite[\href{https://stacks.math.columbia.edu/tag/01PR}{\S 01PR}]{stacks-project}.  For a proper algebraic space $X$ over the spectrum $S=\operatorname{Spec}A$ of a noetherian ring $A$, the following equivalent for a line bundle $L$ on $X$:
\begin{enumerate}[label=(\alph*)]
\item $L$ is ample.

\item Given any coherent sheaf $F$ on $X$, there exists a positive integer $m_1=m_1(F)$ having the property that 
$$
H^i(X,F\otimes L^{\otimes m})=0 \ \ \ \text { for all } i>0, \ m\ge m_1(F).
$$

\item For any coherent sheaf $F$ on $X$, there exists a positive integer $m_2=m_2(F)$ such that $F\otimes L^{\otimes m}$ is generated by its global sections for all $m\ge m_2(F)$.

\item There is a positive integer $m_3>0$ such that $L^{\otimes m}$ is very ample over $S$ for every $m\ge m_3$.
\end{enumerate}

The standard proofs in \cite[Thm.~1.2.6]{Laz04I} or \cite{hartshorne} can easily be adapted to the case of algebraic spaces.  Using this, one obtains from the standard arguments (e.g., \cite[1.2.23]{Laz04I}), as explained in \cite[Thm.~3.11]{kollar_complete_moduli}, the Nakai--Moishezon--Kleiman criterion for ampleness, namely:
 \emph{Given 
be a proper algebraic space $X$ over $\mathbb C$ and a line bundle $L$ on $X$,   then $L$ is ample if and only if for every irreducible closed subspace $V\subseteq X$ one has $\int_Vc_1(L)^{\dim V}>0$.}

From these results, using the standard arguments (e.g., \cite[Prop.~1.2.13 and  Cor.~1.2.28]{Laz04I}), one obtains:

\begin{lem}[Finite pull-back]\label{L:amp-fin}
Let $f:Y\to X$ be a dominant finite morphism of proper algebraic spaces over $\mathbb C$, and let $L$ be a line bundle on $X$.  Then $L$ is ample if and only if $f^*L$ is ample. 
\end{lem}

\begin{proof}
This follows directly from what is above, using the standard arguments in say  \cite[Prop.~1.2.13 and  Cor.~1.2.28]{Laz04I}. 
Note that in the proof of \cite[Cor.~1.2.28]{Laz04I}, there is no need to take hyperplane sections of $f^{-1}(V)$, since $f$ is assumed to be finite, which is stable by base change, so that $f^{-1}(V)\to V$ is already finite and surjective.  
\end{proof}

One also has the following standard result about relatively ample line bundles.  For a morphism $f:X'\to X$ from an algebraic space to a scheme, we use the standard definition \cite[\href{https://stacks.math.columbia.edu/tag/01VG}{\S 01VG}]{stacks-project} for a line bundle $L'$ on $X'$ to be $f$-relatively ample (or simply $f$-ample), which in particular, tacitly now includes the assumption that for every open affine $U\subseteq X$, the pre-image $X'_U:=f^{-1}(U)$ is a scheme.

\begin{pro}[Ampleness in families]\label{P:amp-fam}
Let $f:X'\to X$ be a proper morphism from an algebraic space to a scheme  of finite type over $\mathbb C$, and let $L'$ be a line bundle on $X'$.  
If $L'|_{X'_x}$ is ample over a fiber $X_x'$, then there is an open neighborhood $U\subseteq X$ of $x$ such that for every $y\in U$, the line bundle  $L'|_{y}$ is ample over $X'_y$. Moreover, $L'$ is $f$-relatively ample if and only if $L'|_{X'_x}$ is ample for every fiber $X'_x$.
\end{pro}

\begin{proof}
With the discussion above, the proof in say \cite[Thm.~1.2.17]{Laz04I}
 caries over directly to algebraic spaces.  Up to taking a tensor power of $L'$, the argument there shows that if $L'|_{X'_x}$ is ample over a fiber $X_x$', then  after restriction to a sufficiently small open affine neighborhood $U\subseteq X$ of $x$,  one has a finite morphism $\phi:X'_U\to \mathbb P^r_U$, with $L'|_{X'_{U}}=\phi^*\mathcal O_{\mathbb P^r_U}(1)$.  We can now then use  the cohomological criterion for flatness, since $X'_U$ is defined over the affine $U$; for any coherent sheaf $F$ on $X'_U$, we have $H^i(X'_U,F\otimes L'|_{X'_{U}}^{\otimes m})= H^i(\mathbb P^r_U,\phi_*F\otimes \mathcal O_{\mathbb P^r_U}(m))=0$ for $m\gg 0$.  Therefore $L'|_{X'_{U}}$ is ample on $X'_U$.  The rest of the proposition follows immediately. 
 \end{proof}

\fi 
%%%%%%%%%%%%%%%%%%%%%%%%%%%%%%%%%%%%%%%%%%%%%%%%%%%%%%%%%%%%%%%%%%%%%%%%%%%%%%%%%%%%%% END AXRIV IF %%%%%%%%%%%%%%%%%%%%%%%%%%%%%%%%%%%%%%%%%%%%%%%%%%%%%%%%%%%%%

 \bibliographystyle{amsalpha}
 \bibliography{mhm_bib}

\end{document}